\documentclass[11pt, reqno]{amsart}

\usepackage{amsmath,amssymb,amsthm,amsfonts,mathrsfs,bigints,paralist,graphics,graphicx,epsfig,epstopdf,cite,upgreek,xcolor,boites,caption,soul}
\usepackage{mathtools}
\mathtoolsset{showonlyrefs,showmanualtags}
\usepackage{hyperref}
\hypersetup{hidelinks}
\usepackage{placeins}
\usepackage[pagewise]{lineno}

\usepackage{geometry}
\usepackage{bbm}

\theoremstyle{plain}
\newtheorem{theorem}{Theorem}[section]
\newtheorem{corollary}[theorem]{Corollary}
\newtheorem{lemma}[theorem]{Lemma}
\newtheorem{proposition}[theorem]{Proposition}

\theoremstyle{definition}
\newtheorem{definition}[theorem]{Definition}

\theoremstyle{remark}
\newtheorem{remark}[theorem]{Remark}

\newcommand{\ep}{\varepsilon}

\newcommand{\sgn}{\text{sgn}}
\newcommand*\di{\mathop{}\!\mathrm{d}}
\newcommand{\bb}[1]{\mathbb{#1}}
\newcommand{\al}[1]{\mathcal{#1}}

\allowdisplaybreaks

\begin{document}
\title[Wasserstein geometry of nonnegative measures on finite Markov chains I]{Wasserstein geometry of nonnegative measures on finite Markov chains I: Gradient flow}

\author[Q. Mao, X. Wang, and X. Xue]{Qifan Mao, Xinyu Wang, and Xiaoping Xue$^{*}$}

\address{School of Mathematics, Harbin Institute of Technology, Harbin  150001, People's Republic of China}
\email{qifanmao@stu.hit.edu.cn}
\email{wangxinyumath@hit.edu.cn}
\email{xiaopingxue@hit.edu.cn}
\begin{abstract}
We investigate a Benamou--Brenier type transportation metric for nonnegative measures on a finite reversible Markov chain, which endows the space of measures with a Riemannian structure.
Using this geometric framework, we identify a generalized heat equation with source as the gradient flow of the discrete entropy.
Moreover, by means of a local \L{}ojasiewicz inequality, we prove exponential convergence of the flow to a unique equilibrium.
Our results clarify the role of the Benamou--Brenier formulation in discrete optimal transport for nonnegative measures and provide a coherent geometric interpretation of generalized diffusion equations with source terms.
\end{abstract}
\thanks{$^{*}$Corresponding author. }
\maketitle

\setcounter{equation}{0}

\section{Introduction}
Moving mass on a graph or network plays a central role in a wide range of disciplines.
Typical examples are probability diffusion on graphs, flows of goods in transportation and logistics networks, 
    and population processes with birth--death events on contact networks \cite{lovasz1993random,chung1997spectral,farahani2013review,pastor2015epidemic}.
However, many graph models do not conserve total mass: 
    nodes may produce or consume material, external inflows or outflows can act at selected locations, and reaction or birth--death mechanisms can alter the amount of mass. 
Such nonconservative effects are standard in network models of epidemics and population dynamics, as well as in reaction--diffusion systems on graphs \cite{pastor2015epidemic,zuniga2020reaction,weber2006multicomponent}. 
While optimal transport on graphs has been extensively studied, most existing formulations restrict attention to probability measures, thereby enforcing conservation of total mass \cite{chow2012fokker,maas2011gradient,mielke2011gradient}. As a result, these probabilistic formulations cannot fully capture the nonconservative behaviors described above. To overcome this limitation, we develop in this paper a Benamou--Brenier--type transportation metric for general nonnegative measures on finite graphs, without imposing total-mass conservation. This generalized formulation naturally accommodates nonconservative processes and enables a gradient-flow interpretation of a corresponding generalized heat equation on graphs.

To set the stage, we first recall the optimal transport theory on probability measures of general metric spaces.
Two tools are particularly relevant here.
First, the Benamou--Brenier formulation turns the quadratic Wasserstein distance into a convex action minimization along time–dependent flows, 
    linking paths of measures with flux and continuity \cite{benamou2000computational}.
It measures the minimal effort to rearrange one density into another under the continuity equation, 
    so it captures both cost and kinematics of transport.
Second, Otto’s calculus reads several diffusion equations as gradient flows of the entropy in Wasserstein space, 
    providing a simple framework for evolution and dissipation \cite{otto2001geometry}.
We also note that the Jordan--Kinderlehrer--Otto time–discrete scheme is a classical route to diffusion in the conservative case \cite{jordan1998variational}.
These viewpoints are presented systematically in standard references on gradient flows and optimal transport \cite{ambrosio2005gradient,villani2003topics,villani2008optimal}. 

While one may apply metric–space transport frameworks to a graph once a path metric is chosen, this “lift” does not automatically align with the discrete calculus on graphs. 
In 2011, Maas constructed a discrete Benamou--Brenier transport metric on the probability simplex, so that the heat flow on a Markov chain evolves as the entropy gradient flow \cite{maas2011gradient}. 
Erbar and Maas introduced a notion of Ricci curvature via geodesic convexity of the entropy, yielding discrete analogues of Lott--Sturm--Villani and links to functional inequalities \cite{erbar2012ricci}. 
Further developments include geodesic convexity of the relative entropy and a detailed analysis of discrete geodesics on finite Markov chains \cite{mielke2013geodesic,erbar2019geometry}. 
On the PDE side, graph-based Fokker--Planck equations have been cast as free-energy gradient flows with convergence estimates and entropy dissipation, and nonlocal interaction models on graphs have been linked to gradient-flow structures and continuum limits \cite{chow2012fokker,che2016convergence,chow2018entropy,CarrilloE2025,CarrilloWang2025,esposito2021nonlocal}.
Related results on metric graphs establish a Benamou--Brenier formula and identify McKean--Vlasov dynamics as gradient flows \cite{erbar2022gradient}, 
    while Łojasiewicz-type tools yield decay rates for free energies on graphs and have recent extensions \cite{li2022lojasiewicz,li2025gradient}. 
\newline 

For nonnegative measures, the corresponding transport structures have only been explored systematically since about a decade ago.
In 2014, Piccoli and Rossi introduced generalized Wasserstein distances, 
    an interpolation between the classic Wasserstein distance and the total variation distance, 
    allow comparison of measures of unequal mass, 
    and were used to study transport equations with sources \cite{piccoli2014generalized}. 
A follow–up work established the Benamou–Brenier–type dynamic characterization tailored to the generalized setting \cite{piccoli2016properties}.
Other important works include 
    optimal entropy--transport and the Hellinger--Kantorovich distance with a full geometric--variational theory \cite{liero2018optimal}, 
    equivalent dynamic and Kantorovich formulations that couple reaction with transport \cite{chizat2018interpolating},  
    related interpolating distances and domain-based constructions for gradient flows \cite{dolbeault2009new,figalli2010new}, 
    and a reaction--transport metric on finite Radon measures with links to Otto’s calculus \cite{kondratyev2016new}. 
On the computational side, unnormalized optimal transport supplies tractable extensions and practical algorithms for unequal-mass data \cite{gangbo2019unnormalized}.

For graph models without mass conservation, one needs a geometry on \emph{nonnegative} measures, not only on probabilities.
Recent work offers computational and statistical tools for unequal-mass data on graphs, 
    including unbalanced Sobolev transport and entropic transportation on random graphs \cite{le2023scalable,keriven2023entropic}.
These advances emphasize estimation and scalability, 
    but an intrinsic, graph–adapted framework remains to be developed. Therefore, we consider the following questions:
    
\begin{itemize}
    \item[(Q1)] What is a natural, graph–adapted transport metric on finite graphs for \emph{nonnegative} measures?
    \item[(Q2)] Can such a metric support PDE analysis on graphs, for instance by enabling gradient–flow formulations and long–time estimates?
\end{itemize}
\vspace{0.2cm}

The main purpose of this paper is to answer these questions. 
For (Q1), we introduce a Benamou–Brenier setup on finite reversible Markov chains
    that acts directly on nonnegative measures and remains compatible with the standard graph calculus.
Our construction is inspired by two threads.
On the one hand, Piccoli and Rossi’s generalized Wasserstein distances weight transport and mass variation with positive coefficients $(a,b)$ 
    and thus \emph{balance} pure transport against pure creation/removal \cite{piccoli2014generalized,piccoli2016properties}.
On the other hand, Maas and Erbar built a discrete transport metric on graphs by pairing a Benamou--Brenier action with the discrete gradient–divergence structure \cite{maas2011gradient,erbar2012ricci}. Combining these ideas,
    we define a distance by minimizing a transport–source action subject to a nonconservative continuity equation.
For two nonnegative measures $\mu_0,\mu_1$ on the finite state space, we define their Benamou--Brenier metric as follows:
    \begin{equation*}
    \begin{gathered}
    \al W_{p}^{a,b}(\mu_0,\mu_1)^{2}
    := \inf\left\{
    a^{2}\!\int_{0}^{1} h_t^{2}\di t
    \;+\;
    b^{2}\!\int_{0}^{1} \|\nabla \psi_t\|_{\mu_t}^{2}\di t
    \right\},
    \end{gathered}
    \end{equation*}
    where the infimum runs over curves $t\mapsto\mu_t$ joining $\mu_0$ to $\mu_1$, potentials $\psi_t$, and scalar source rates $h_t$ that solve
    \begin{equation*}
    \left\{
    \begin{aligned}
    & \dot{\mu}_t + \nabla\!\cdot\big(\,\hat{\mu}_t * \nabla \psi_t\,\big) \;=\; h_t\,p, \\
    & \mu_{t=0}=\mu_0,\qquad \mu_{t=1}=\mu_1 .
    \end{aligned}
    \right.
    \end{equation*}
Here 
    $\|\nabla \psi_t\|_{\mu_t}$ is the graph–weighted flux norm \eqref{eq_mu_2norm} induced by $\mu_t$,
    $\hat{\mu}_t$ denotes the mobility associated with $\mu_t$ as in \eqref{eq_hatmu}, 
    and $\ast$ is the Hadamard product defined by $(\hat{\mu}_t \ast \nabla \psi_t) (x,y)=\hat{\mu}_t(x,y) \nabla \psi_t(x,y)$.
The coefficients $a>0$ and $b>0$ weigh the source and transport costs, in the spirit of \cite{piccoli2014generalized,piccoli2016properties}.
Yet unlike \cite{piccoli2014generalized,piccoli2016properties} 
    where the source may vary freely in space and time,
    the \emph{direction} of mass variation is prescribed by a fixed strictly positive probability vector $p$,
    so that creation/removal is encoded by a single scalar amplitude $h_t$.
\newline 

For (Q2), following Otto’s calculus, 
    we endow the cone of strictly positive measures with a first-order structure.
Given tangent vectors $\rho,\xi$ at a strictly positive measure $\mu$, 
    there exist unique pairs $(\nabla\psi_\rho,h_\rho)$ and $(\nabla\psi_\xi,h_\xi)$ solving
    \begin{equation*}
    \rho + \nabla_{\!\mu}\!\cdot\nabla\psi_\rho = h_\rho\,p,
    \qquad
    \xi + \nabla_{\!\mu}\!\cdot\nabla\psi_\xi = h_\xi\,p .
    \end{equation*}
We then define the inner product
    \begin{equation*}
    \mathbf{g}_\mu(\rho,\xi) := a^{2}\,h_\rho h_\xi \;+\; b^{2}\,\langle \nabla\psi_\rho, \nabla\psi_\xi\rangle_\mu .
    \end{equation*}
Here, $\langle~\cdot,\cdot ~\rangle_\mu$ is the $\mu$-dependent inner product defined in \eqref{inner_mu}. 
Under this geometry, the discrete entropy
    \begin{equation*}
    \bb H(\mu) := \sum_{x\in \al X}\big(\mu(x)\log\mu(x)-\mu(x)\big)\,\varpi(x)
    \;=:\; \langle \mu,\,\log\mu-1\rangle_{\varpi}
    \end{equation*}
    generates a gradient flow given by
    \begin{equation*}
    \dot{\varrho}_t \;=\; b^{-2}\,\Delta \varrho_t \;-\; a^{-2}\,\langle \log \varrho_t,\,p\rangle_{\varpi}\,p ,
    \end{equation*}
    where $\Delta$ is the graph Laplacian \eqref{eq_Delta=K-I} associated with the chain and $\varpi$ is the steady state.
Using Łojasiewicz’s inequality, 
    solutions converge exponentially fast to the unique equilibrium  $\mathbf{1}=(1,...,1)$ (Theorem \ref{thm_heatflow_converge}).
\newline

Within this landscape, our main contributions are twofold. 
First, we construct on a finite reversible Markov chain a {complete} Benamou--Brenier distance for nonnegative measures. 
Second, we show that the entropy generates a well–posed gradient flow with respect to this distance, and demonstrate that it will converge to a unique equilibrium point exponentially. 
In a follow-up paper \cite{mao2026wasserstein2}, we investigate the associated geodesic problem and establish a corresponding duality formula, completing the geometric and variational picture.
\newline 

Our paper is organized as follows.
In Section \ref{sec_preliminaries}, we introduce the basic setup of the Markov chain, 
    including notation and definitions that will be extensively used in the following sections.
In Section \ref{sec_space}, we verify the metric axioms of the new transport metric.
In Section \ref{sec_flow}, we establish the associated Riemannian structure 
    and show that the semigroup of a certain heat equation is the gradient flow of the entropy functional.
Finally, Section \ref{sec_conc} is devoted to a brief summary of our main results and some discussions on remaining issues for future work. 

\section{Basic settings of the Markov chain}
\label{sec_preliminaries}

In this section, we introduce the Markov chain and several notations, including discrete divergence, Laplacian, inner product, logarithmic mean, and so on.\newline

We consider a finite reversible Markov chain and fix the notation used throughout, following the conventions of \cite{maas2011gradient,erbar2012ricci}. Let $\al X$ be a finite space with $N$ points, 
    and let $K: \al X \times\al X \to \mathbb{R}_+$ be an irreducible Markov kernel. 
By standard results in Markov chain theory, 
    there exists a unique steady-state distribution $\varpi: \al X \to \mathbb{R}_+$,
    meaning that
    \begin{equation*}
    \varpi(x)= \sum_{y\in \al X}  \varpi(y)K(y,x),
    \ \forall~ x\in \al X,
    \quad
    \text{and} \quad
    \sum_{y\in \al X}\varpi(y)=1.
    \end{equation*}
Assume that $\varpi$ is reversible for $K$, which means
\begin{equation}
    \label{eq_markov_reversible}
    K(x,y)\varpi(x)=K(y,x)\varpi(y),\quad \forall~ x,y\in \al X.
\end{equation}
We denote the set of all nonnegative densities by  
\begin{equation*}
    \al M=\al M(\al X) := \{ \mu: \al X \to \mathbb{R} \mid \forall ~x \in \al X, ~\mu(x) \geq 0 \},
\end{equation*}
and the strictly positive densities constitute
\begin{equation}\label{Mani}
    \al M_+=\al M_+(\al X):=\{\mu:\al X\rightarrow \bb R \mid \forall~ x \in \al X,~\mu(x)> 0\}.
\end{equation}

Since all real-valued mappings on $\al X$ are finite arrays of length $N$, 
    we use $\mathbb{R}^{\al X}$ and $\mathbb{R}^{\al X\times \al X}$ for 
    real–valued functions on vertices and on ordered pairs of vertices,
    and freely identify these spaces with $\mathbb{R}^{N}$ and $\mathbb{R}^{N\times N}$. 
We write $[\ \cdot,\cdot\ ]$ for the standard Euclidean inner product and $|\cdot|_2$ for the Euclidean norm. 

Let $\psi$ be an arbitrary element in $\mathbb{R}^{\al X}$.
Its {$\mathrm L^1$-norm} with respect to $\varpi$ is denoted by
    \begin{equation}\label{L_1norm}
    \| \psi \|_{\varpi,1} := [|\psi|, \varpi] = \sum_{x \in \al X} \varpi(x) |\psi(x)|.
    \end{equation}
For a nonnegative density $\mu\in \al M$,
    $\|\mu\|_{\varpi,1}$ is interpreted as the total mass of $\mu$. 
The {discrete gradient} of $\psi$ is given by  
    \begin{equation*}
    \nabla \psi (x,y) := \psi(y)-\psi(x), \quad \forall ~x,y \in \al X.
    \end{equation*}  
Let $\Psi: \al X \times\al X \to \mathbb{R}$ be a real–valued function on ordered pairs of vertices. 
Its {discrete divergence} is defined as  
    \begin{equation*}
    (\nabla \cdot\Psi)(x) := \frac{1}{2} \sum_{y \in \al X}  (\Psi(x,y) - \Psi(y,x))K(x,y),
    \quad \forall~ x\in \al X.
    \end{equation*}
The discrete Laplacian of $\psi$ is $\Delta\psi:=\nabla\cdot\nabla\psi$.
We have the identity $\Delta=K-Id$, 
    because for any $x\in \al X$,
    \begin{equation}
    \label{eq_Delta=K-I}
    \begin{aligned}
    \Delta\psi(x) = (\nabla\cdot\nabla\psi)(x)
    =&\
    \frac{1}{2}\sum_{y\in \al X}
    \big(\nabla\psi(x,y)
    -
    \nabla\psi(y,x)
    \big)
    K(x,y)
    \\
    =&\
    \sum_{y\in \al X}
    \big(\psi(y)-\psi(x)
    \big)
    K(x,y)
    \\
    =&\
    \sum_{y\in \al X}
    K(x,y)\psi(y)-\sum_{y\in \al X}K(x,y)\psi(x)
    \\
    =&\
    \sum_{y\in \al X}
    \big(K(x,y)-Id(x,y)\big)\psi(y)
    \\
    =&\
    \big((K-Id)\psi\big)(x).
    \end{aligned}
    \end{equation}
Here $Id$ is the identity matrix and we use that $K$ is a row-stochastic matrix.

We define the $\varpi$-weighted positive semidefinite symmetric bilinear form for $\varphi, \psi\in \bb R^{\al X}$ and $\Phi, \Psi\in\bb R^{\al X\times \al X}$:  
\begin{equation*}
\begin{gathered}
    \langle \varphi, \psi \rangle_{\varpi} := \sum_{x \in \al X} \varphi(x) \psi(x) \varpi(x),\\
    \langle \Phi, \Psi \rangle_{\varpi} := \frac{1}{2} \sum_{x,y \in \al X} \Phi(x,y) \Psi(x,y) K(x,y) \varpi(x).
\end{gathered}
\end{equation*}  
With these definitions in place, we obtain the integration by parts formula:
    \begin{equation}
    \label{eq_int_by_parts}
    \begin{aligned}
    \langle \nabla\psi, \Psi \rangle_\varpi
    =&
    \frac{1}{2} \sum_{x,y \in \al X} (\psi(y)-\psi(x))\Psi(x,y) K(x,y) \varpi(x)
    \\
    =&
    \frac{1}{2} \sum_{x,y \in \al X} \psi(y)\Psi(x,y) K(x,y) \varpi(x)
    -
    \frac{1}{2} \sum_{x,y \in \al X} \psi(x)\Psi(x,y) K(x,y) \varpi(x)
    \\
    =&
    \frac{1}{2} \sum_{x,y \in \al X} \psi(x)\Psi(y,x) K(y,x) \varpi(y)
    -
    \frac{1}{2} \sum_{x,y \in \al X} \psi(x)\Psi(x,y) K(x,y) \varpi(x)
    \\
    =&
    \frac{1}{2} \sum_{x,y \in\al X} \psi(x)\Psi(y,x) K(x,y) \varpi(x)
    -
    \frac{1}{2} \sum_{x,y \in\al X} \psi(x)\Psi(x,y) K(x,y) \varpi(x)
    \\
    =&
    -\sum_{x\in \al X} \psi(x)
    \left(\frac{1}{2} \sum_{y\in\al X}(\Psi(x,y)-\Psi(y,x)) K(x,y)\right)
    \varpi(x)
    = - \langle \psi, \nabla \cdot \Psi \rangle_\varpi,
    \end{aligned}
    \end{equation}
    where we used the reversibility of $K$ and $\varpi$ as in \eqref{eq_markov_reversible}.

The logarithmic mean $\theta$ is defined for $u,v\in \mathbb{R}_+$ by
    \begin{equation}
    \label{eq_logarithmic_mean}
    \theta(u,v):= \int_0^1 u^\xi v^{1-\xi} \di \xi
    =
    \left\{
    \begin{aligned}
    & v, && \text{if} \ u=v, \\
    & \frac{u-v}{\log u -\log v}, && \text{if} \ u\neq v \ \text{and} \ u,v>0, \\
    & 0, && \text{if} \ u=0 \ \text{or} \ v=0.
    \end{aligned}
    \right.
    \end{equation}
$\theta$ is nonnegative, symmetric, and strictly increasing in each argument on $(0,\infty)\times(0,\infty)$.

While \cite{maas2011gradient,erbar2012ricci} introduced a broader class of functions $\theta$ for a general theory, 
    the logarithmic choice is particularly significant for the analysis of entropy. 
We fix this logarithmic choice throughout, but retain the notation $\theta$ for ease of reference.

For $\mu \in \al M$, we use the notation $\hat{\mu}:\al X \times \al X \to \mathbb{R}$ as  
    \begin{equation}
    \label{eq_hatmu}
    \hat{\mu}(x,y) := \theta(\mu(x), \mu(y)),
    \quad\forall~ x,y\in \al X.
    \end{equation}
Using $\hat{\mu}$, we define a $\mu$-dependent positive semidefinite symmetric bilinear form for $\Phi,\Psi\in\mathbb{R}^{\al X \times \al X}$:  
    \begin{equation}\label{inner_mu}
    \langle \Phi, \Psi \rangle_{\mu}:=\langle \Phi, \hat{\mu}* \Psi \rangle_{\varpi}
    =
    \frac{1}{2} \sum_{x,y \in \al X} \Phi(x,y) \Psi(x,y) \hat{\mu}(x,y) K(x,y) \varpi(x)
    ,
    \end{equation}
    where for $x,y\in \al X$, $\hat{\mu}* \Psi(x,y)= \hat{\mu}(x,y)\Psi(x,y)$.
The corresponding seminorm is
    \begin{equation}
    \label{eq_mu_2norm}
    \|\Phi\|_{\mu} := \sqrt{\langle \Phi, \Phi \rangle_{\mu}}.
    \end{equation}

We declare an equivalence relation on $\mathbb{R}^{\al X\times \al X}$ 
    by saying that two functions are equivalent if and only if they coincide at every pair $(x,y)$ with $\hat\mu(x,y)K(x,y)>0$.
Denote by $\mathscr{G}_\mu$ the resulting collection of equivalence classes.
In the degenerate case where $\hat\mu(x,y)K(x,y)=0$ for every $(x,y)$, 
    the quotient $\mathscr{G}_\mu$ reduces to a singleton.
When the meaning is clear, for $\Phi\in \mathbb{R}^{\al X\times \al X}$,
    we use the same symbol $\Phi$ for a representative of its class in $\mathscr{G}_\mu$.
Since $\langle\cdot,\cdot\rangle_\mu$ depends only on the values at those pairs where $\hat\mu(x,y)K(x,y)>0$, 
    it descends to an inner product on $\mathscr{G}_\mu$.
In the finite–state setting, $\mathscr{G}_\mu$ is therefore a Hilbert space.

Let $\mathrm L^2(\al X,\varpi)$ be the Hilbert space $\mathbb{R}^{\al X}$ endowed with $\langle\cdot,\cdot\rangle_\varpi$.
The gradient operator $\nabla:\mathrm L^2(\al X,\varpi)\to \mathbb{R}^{\al X\times\al X}$ descends, 
    via the $\mu$–dependent quotient that defines $\mathscr{G}_\mu$, to a well–defined linear map $\nabla:\mathrm L^2(\al X,\varpi)\to\mathscr{G}_\mu$. 
We henceforth regard $\nabla \psi$ as its class in $\mathscr{G}_\mu$.
Thus the negative adjoint of $\nabla$, 
    the $\mu$–divergence operator $\nabla_\mu\cdot:\mathscr{G}_{\mu}\to \mathrm L^2(\al X,\varpi)$, is given by
    \begin{equation*}
    (\nabla_\mu \cdot \Psi)(x)
    := (\nabla\cdot(\hat\mu* \Psi))(x)
    = \frac{1}{2}\sum_{y\in \al X}\big(\Psi(x,y)-\Psi(y,x)\big)\hat{\mu}(x,y)K(x,y),\quad \forall~ x\in \al X.
    \end{equation*}

For $\mu\in \al M$, the matrices $A_\mu,B_\mu\in \mathbb{R} ^{\al X \times \al X}$ are defined by
    \begin{equation}
        \label{def_Amu}
        A_\mu(x,y):=
        \left\{
            \begin{aligned}
                & \sum_{z\ne x} K(x,z) \hat\mu (x,z) \varpi (x),&\text{if }x=y, \\
                &-K(x,y) \hat \mu(x,y) \varpi(x), & \text{if }x\ne y,
            \end{aligned}
        \right.
    \end{equation}
and
    \begin{equation}
        \label{def_Bmu}
        B_{\mu}(x,y):=
        \left\{
            \begin{aligned}
                & \sum_{z\ne x} K(x,z) \hat\mu (x,z),  &\text{if } x=y, \\
                &-K(x,y) \hat \mu(x,y), &\text{if } x\ne y.
            \end{aligned}
        \right.
    \end{equation}
Let $\Pi:=\mathrm{diag}\varpi\in \mathbb R^{\al X \times \al X}$.
Then $A_\mu=\Pi B_\mu$.
Moreover we have for any $\psi \in \mathbb{R}^\al X$,
    \begin{equation}
    \label{eq_matrix_reform}
    \begin{gathered}
    \nabla_{\mu}\cdot\nabla\psi=\nabla\cdot (\hat \mu * \nabla \psi)=- B_\mu \psi, 
    \quad
    \| \nabla\psi\|_\mu^2=\langle \nabla\psi,\nabla\psi\rangle_\mu=[A_\mu\psi,\psi].
    \end{gathered}  
    \end{equation}

\section{Metric space $(\mathcal{M},\mathcal{W}^{a,b}_p)$}
\label{sec_space}
In this section, we introduce $\mathcal{W}^{a,b}_p$ as a Benamou--Brenier-type metric, 
    explore the existence of admissible curves that satisfy the continuity equation with source
    and verify the axioms of metric for $\mathcal{W}^{a,b}_p$.

\subsection{Definition of $\al W^{a,b}_p$}

From now on, we fix $a, b > 0$ and 
    let $p$ be a fixed density in $\al M_+$ with $\| p\|_{\varpi,1}=1$,
    so that $p$ is a probability density.

\begin{definition}[Benamou--Brenier-type metric on $\al M$]
\label{def_Wpab}
We define for any $\mu_0,\mu_1\in \al M$,
    \begin{equation}
    \label{eq_def_w}
    \al {W}^{a,b}_p(\mu_0, \mu_1)^2 := 
    \inf \left\{ E^{a,b}_{\rm Quad}((\mu_t, \psi_t, h_t)_{t\in[0,1]}) 
    \mid (\mu_t, \psi_t, h_t)_{t\in[0,1]} \in \mathrm{CE}_p(\mu_0, \mu_1;[0,1]) \right\}.
    \end{equation}
Here, the admissible trajectory set $\text{CE}_p(\mu_0, \mu_1; [0,1])$ consists of all triples $(\mu_t, \psi_t, h_t)_{t\in[0,1]}$ such that 
    \begin{itemize}
    \item[$\circ$] $(\mu_t, \psi_t, h_t)_{t\in[0,1]}$ is sufficiently regular,
        with $\mu_t\in \mathrm{A\!C}([0,1];\al M)$,
        $\psi_t:[0,1]\to \mathbb{R}^{\al X}$ measurable,
        $h_t:[0,1]\to \mathbb{R}$ measurable;
    \item[$\circ$] $(\mu_t, \psi_t, h_t)_{t\in[0,1]}$ satisfies the continuity equation: 
        for almost every $t\in[0,1]$,
        \begin{equation}
        \label{eq_ce}
        \dot\mu_t(x) + \sum_{y\in \al X} (\psi_t(y)-\psi_t(x))\hat\mu_t(x,y)K(x,y)
        = h_t p(x);
        \end{equation}
    \item[$\circ$] $(\mu_t, \psi_t, h_t)_{t\in[0,1]}$ satisfies the end-point condition:
        \begin{equation*}
        \mu_t |_{t=0} = \mu_0, \quad \mu_t |_{t=1} = \mu_1.
        \end{equation*}
    \end{itemize}
The functional $E_{\rm Quad}^{a,b}$ is given by
    \begin{equation}
    \label{eq_def_E2_origin}
    E^{a,b}_{\rm Quad} ((\mu_t, \psi_t, h_t)_{t\in[0,1]}) 
    :=
    a^2 \int_0^1 h_t^2
    \mathrm d t 
    + 
    b^2 \int_0^1 
    \frac{1}{2}\sum_{x,y\in \al X} (\psi_t(y)-\psi_t(x))^2\hat\mu_t(x,y)K(x,y)\varpi(x)
    \mathrm{d}t.
    \end{equation}
The subscript ``Quad'' here indicates integrating a quadratic integrand in time.
\end{definition}

Using the discrete calculus symbols introduced in Section \ref{sec_preliminaries},
    the energy functional $E^{a,b}_{\rm Quad}$ and the continuity equation \eqref{eq_ce} admit a reformulation that is formally closer to the classic Benamou--Brenier formula:
    we have for $(\mu_t, \psi_t, h_t)_{t\in[0,1]}\in \mathrm{CE}_p(\mu_0,\mu_1;[0,1])$,
    \begin{equation}
    \label{eq_def_E2_discrete_calculus}
    \begin{gathered}
    E^{a,b}_{\rm Quad} ((\mu_t, \psi_t, h_t)_{t\in[0,1]}) = a^2 \int_0^1 h_t^2
    \mathrm d t 
    + 
    b^2 \int_0^1 
    \| \nabla \psi_t\|_{\mu_t}^2
    \mathrm{d}t
    = a^2 \int_0^1 h_t^2
    \mathrm d t 
    + 
    b^2 \int_0^1 
    \langle \nabla \psi_t, \nabla \psi_t \rangle_{\mu_t}
    \mathrm{d}t.
    \end{gathered}
    \end{equation}
And the continuity equation \eqref{eq_ce} becomes
    \begin{equation}
    \label{eq_ce_discrete_calculus}
    \dot\mu_t + \nabla_{\!\mu_t}\!\cdot\!\nabla\psi_t
    = h_t p,
    \qquad
    \text{or}
    \ \
    \dot\mu_t + \nabla\!\cdot\!(\hat\mu_t \ast \nabla\psi_t)
    = h_t p.
    \end{equation}
We also have the matrix form of $E^{a,b}_{\rm Quad}$ 
    and the continuity equation thanks to \eqref{eq_matrix_reform}:
    \begin{gather}
    \label{eq_def_E2_matrix}
    E^{a,b}_{\rm Quad} ((\mu_t, \psi_t, h_t)_{t\in[0,1]})=a^2\int_0^1  h_t^2 \di t
    +  b^2\int_0^1[A_{\mu_t}\psi_t,\psi_t] \di t, \\
    \label{eq_ce_matrix}
    \dot \mu_t=B_{\mu_t}\psi_t + h_t p.
    \end{gather}

\subsection{Solvability criteria for the continuity equation}
The matrix formulation \eqref{eq_ce_matrix} is convenient for proving solvability of the continuity equation in this part.
We shall record in Lemma~\ref{lem_AB_ker_ran} the kernel--range structure of these operators, 
    and in Lemma~\ref{lem_B_mu_isormorphism} the associated isomorphism properties on the relevant subspaces.
With these ingredients, Lemma~\ref{lem_ce_solver} provides a solvability criterion for the source–free equation at strictly positive reference states, 
    identifying the solution set up to constants and the uniquely determined gradient field.
Finally, Lemma~\ref{lem_source_ce_solver} extends the solvability criterion to the equation with source.

Let $\mu\in \al M$. We introduce an equivalence relation on $\al X$ indicating connectivity in the Markov chain.
For $x,y\in \al X$, by $x\sim_\mu y$ we mean that
    \begin{enumerate}
    \item $x=y$ or
    \item there exist $k\ge 1$ and $x_1,...,x_k\in \al X$ such that
    \begin{equation*}
    \hat\mu(x,x_1)K(x,x_1),\,
    \hat\mu(x_1,x_2)K(x_1,x_2),\,
    ...,\,
    \hat\mu(x_k,y)K(x_k,y)>0.
    \end{equation*}
    \end{enumerate}

In the next lemma we treat $A_\mu$ and $B_\mu$ as linear transformations on $\mathbb{R}^{\al X}$.
A proof is given in \cite[Lemma 3.17]{erbar2012ricci}.
\begin{lemma}
    \label{lem_AB_ker_ran}
    For $\mu\in \al M$ we have
    \begin{equation*}
    \begin{aligned}
    \mathrm{Ker} A_\mu = \mathrm{Ker} B_\mu 
    & =
    \{
    \psi \in \bb R ^{\al X} ~|~ \psi (x) =\psi (y) \text{~whenever~} x\sim_{\mu} y
    \},
    \\
    \mathrm{Ran} A_\mu
    & =
    \left\{
    \psi \in \bb R ^{\al X} ~\bigg|~ \forall x\in \al X, \sum_{y\sim_{\mu}x} \psi(y) =0
    \right\},
    \\
    \mathrm{Ran} B_\mu
    & =
    \left\{
    \psi \in \bb R ^{\al X} ~\bigg|~ \forall x\in \al X, \sum_{y\sim_{\mu}x} \psi(y) \varpi (y)=0
    \right\}.
    \\
    \end{aligned}
    \end{equation*}
\end{lemma}

Below is a basic structural fact for the operators $A_\mu$ and $B_\mu$ on their natural range spaces.
For a proof, consult Lemma 3.18 in \cite{erbar2012ricci}.
\begin{lemma}
    \label{lem_B_mu_isormorphism}
    For $\mu \in \al M$, the operator $A_\mu:\psi\mapsto A_\mu \psi$ is self-adjoint with respect to the inner product $[\cdot,\cdot]$ on $\mathbb{R}^{\al X}$ 
    and the restricted operator $A_\mu\!\!\restriction_{\mathrm{Ran} A_\mu}$ is an isomorphism on $\mathrm{Ran} A_\mu$.
    Furthermore, the restriction $B_{\mu}\restriction_{\mathrm{Ran} A_\mu}$ is an isomorphism $\mathrm{Ran}A_\mu\to\mathrm{Ran}B_\mu$.
\end{lemma}

Recall that $\al M_+:=\{\mu:\al X\rightarrow \bb R \mid \forall~ x \in \al X,~\mu(x)> 0\}$ is the set of strictly positive densities.
For all  $\mu\in \al M_+$, 
    the connectivity equivalence class collapses to the single class $\al X$,
    as $K$ is irreducible and $\hat \mu$ is always positive.
Then the kernels and ranges described in Lemma \ref{lem_AB_ker_ran} also reduce to the simplest form:
    \begin{equation*}
    \begin{aligned}
    \mathrm{Ker} A_\mu = \mathrm{Ker} B_\mu =\mathrm{span}\{\mathbf{1}\}
    & :=
    \{
    \psi \in \bb R ^{\al X} \mid\forall~ x,y\in \al X,\ \psi (x) =\psi (y)
    \},
    \\
    \mathrm{Ran} A_\mu=\mathsf{H}
    & :=
    \left\{
    \psi \in \bb R ^{\al X} \bigg| \sum_{x\in \al X} \psi(x) =0
    \right\},
    \\
    \mathrm{Ran} B_\mu = \mathsf{H}_\varpi & :=
    \left\{
    \psi \in \bb R ^{\al X} \bigg|  \sum_{x\in \al X} \psi(x) \varpi (x)=0
    \right\}.
    \\
    \end{aligned}
    \end{equation*}

Next we present the solvability criterion for the source–free continuity equation at strictly positive reference states.
\begin{lemma}
\label{lem_ce_solver}
Fix $\mu \in \al M_+$.
The linear equation 
\begin{equation}
\label{eq_cenos_proto}
\nu = B_\mu \psi
\end{equation}
is solvable if and only if $\nu \in \mathsf{H}_{\varpi}$.
In this case the solution set is
    \begin{equation*}
    \{\psi \in \mathbb{R}^{\al X} \mid \nu = B_\mu\psi\} =
    \{\psi \in \mathbb{R}^{\al X} \mid
    \psi-
    (B_\mu\!\!\restriction_{\mathsf{H}})^{-1}(\nu) \in \mathrm{span}\{\mathbf{1}\}
    \},
    \end{equation*}
    and, in particular,
    \begin{equation*}
    \{\nabla\psi \in \mathbb{R}^{\al X\times \al X} \mid \nu = B_\mu\psi\} =
    \{
    \nabla
    (B_\mu\!\!\restriction_{\mathsf{H}})^{-1}(\nu) 
    \}.
    \end{equation*}
\end{lemma}
\begin{proof}
For $\mu\in \al M_+$ the range $\mathrm{Ran}B_\mu$ coincides with the $\mu$–independent space $\mathsf{H}_{\varpi}$, 
    which proves the solvability criterion.

Fix $\nu\in \mathsf{H}_{\varpi}$.
By Lemma \ref{lem_B_mu_isormorphism},
    $B_\mu\!\!\restriction_{\mathsf{H}}$ is an isomorphism from $\mathsf{H}$ to $\mathsf{H}_{\varpi}$,
    then the inverse operator exists and 
    $\psi_\ast:=(B_\mu\!\!\restriction_{\mathsf{H}})^{-1}(\nu)$ solves \eqref{eq_cenos_proto}.
All solutions are then given by $\psi_* +\mathrm{Ker}B_\mu=\psi_*+\mathrm{span}\{\mathbf{1}\}$.

Finally, since the gradient operator $\psi \mapsto \nabla \psi$ maps every element of $\mathrm{span}\{\mathbf{1}\}$ to the zero gradient field, 
    the set 
    \begin{equation*}
    \{\nabla \psi \in \mathbb{R}^{\al X \times \al X} \mid \nu = B_\mu \psi\}
    \end{equation*}
    contains only a single element $\nabla \psi_\ast $.
\end{proof}

For simplicity,
    we denote
    \begin{equation}
    \label{eq_Jmu_characterization}
    J_\mu:\nu\in \mathsf{H}_{\varpi}
    \mapsto
    \nabla\big((B_\mu\!\restriction_{\mathsf{H}})^{-1}\nu\big)
    \in
    \{\Psi\in\mathbb{R}^{\al X\times \al X} \mid \ \exists\,\psi\in\mathbb{R}^{\al X},\ \Psi=\nabla\psi\}.
    \end{equation}
By Lemma \ref{lem_ce_solver}, when $\mu\in\al M_+$, $J_\mu$ is an isomorphism between its domain and its range. 

\begin{lemma}
\label{lem_source_ce_solver}
Fix $\mu\in \al  M_+$, then for any arbitrary $\rho\in \mathbb{R}^{\al X}$, 
    the linear equation in the unknowns $(\nabla\psi_\rho,h_\rho)$
    \begin{equation}
    \label{eq_tangent_identification_0}
    \rho+\nabla_{\!\mu}\!\cdot\nabla\psi_\rho=h_\rho\,p
    \end{equation}
    has a unique solution
    \begin{equation*}
    (\nabla\psi_\rho,h_\rho)
    :=\big(J_\mu([\rho,\varpi]\,p-\rho),\,[\rho,\varpi]\big).
    \end{equation*}
\end{lemma}
\begin{proof}
Applying $[\cdot,\varpi]$ to \eqref{eq_tangent_identification_0} yields
\begin{equation*}
  [\rho,\varpi]+[\nabla_{\!\mu}\!\cdot\nabla\psi_\rho,\varpi]=h_\rho\,[p,\varpi].
\end{equation*}
By convention $[p,\varpi]=1$, and by discrete integration by parts in \eqref{eq_int_by_parts} we have
\begin{equation*}
  [\nabla_{\!\mu}\!\cdot\nabla\psi_\rho,\varpi]
  =\langle\nabla\!\cdot(\hat\mu\ast\nabla\psi_\rho),\mathbf{1}\rangle_\varpi
  =-\langle\hat\mu\ast\nabla\psi_\rho,\nabla\mathbf{1}\rangle_\varpi
  =0.
\end{equation*}
Hence $h_\rho=[\rho,\varpi]$ is the only choice compatible with \eqref{eq_tangent_identification_0}.
    
Since $[h_\rho p-\rho,\varpi]=[\nabla_{\!\mu}\!\cdot\nabla\psi_\rho,\varpi]=0$, the vector $h_\rho p-\rho$ belongs to
\begin{equation*}
  \mathsf{H}_{\varpi}=\{\nu\in\mathbb{R}^\al X \mid [\nu,\varpi]=0\}.
\end{equation*}
By Lemma~\ref{lem_ce_solver},
    $\nabla\psi_\rho:=J_\mu(h_\rho p-\rho)$
    is the only solution to the equation 
    $\nabla_{\mu}\!\cdot\Psi=h_\rho p-\rho$
    that takes the form $\Psi=\nabla \psi$ for some $\psi\in\mathbb{R}^\al X$.
\end{proof}

For $\mu\in \al  M_+,\rho\in \mathbb{R}^\al X$ we write the decomposition
    \begin{equation}
    \label{eq:Tmu-decomposition}
    D_\mu(\rho)
    :=\big(J_\mu([\rho,\varpi]\,p-\rho),\,[\rho,\varpi]\big).
    \end{equation}

Consequently, for any pair of endpoints $\mu_0,\mu_1 \in \al M$,
    assume that there exists an absolutely continuous curve $\mu_t:[0,1]\to \al M$ such that 
    \begin{itemize}
    \item $\mu_t$ connects $\mu_0,\mu_1$: $\mu_t|_{t=0}=\mu_0,\ \mu_t|_{t=1}=\mu_1$;
    \item for a.e. $t\in [0,1],\  \mu_t \in \al M_+$.
    \end{itemize}
Then whenever $\mu_t \in \al M_+$ and $\dot \mu_t$ exists,
    $(\nabla \psi_t,h_t)= D_{\mu_t}(\dot \mu_t)$ is well-defined and satisfies the continuity equation
    \begin{equation*}
    \dot\mu_t + \nabla_{\!\mu_t}\!\cdot\!\nabla\psi_t
    = h_t p.
    \end{equation*}

The final step in constructing an admissible trajectory in $\mathrm{CE}_{p}(\mu_{0},\mu_{1};[0,1])$ is to produce an absolutely continuous curve $\mu_t:[0,1]\to \al M$ joining $\mu_{0}$ to $\mu_{1}$ that stays in $\al M_{+}$ for almost every time.
The idea is to first lift $\mu_0$ and $\mu_1$ in the $p$-direction until totals match, and then redistribute the mass between the lifted points.
We construct a three-phase path: in the first and third phases we interpolate between the original endpoints and their $p$-lifts, while in the second phase we interpolate between the lifted endpoints.
Without loss of generality, assume $[\mu_{1},\varpi]\ge [\mu_{0},\varpi]$.
Fix $\varepsilon\in(0,1)$ and define
    \begin{equation}
    \label{exm_ce_non_empty}
    \begin{gathered}
    \mu_t^\star:=
    \left\{
    \begin{aligned}
    & \mu_0+3t\big([\mu_1-\mu_0,\varpi]+\varepsilon\big)p,&& 0\le t\le\tfrac{1}{3},\vspace{6pt}\\
    & (2-3t)\Big(\mu_0+\big([\mu_1-\mu_0,\varpi]+\varepsilon\big)p\Big)
      +(3t-1)\big(\mu_1+\varepsilon p \big), 
    && \tfrac{1}{3}< t\le\tfrac{2}{3},\vspace{6pt}\\
    & \mu_1 + (3-3t)\varepsilon p , && \tfrac{2}{3}< t\le 1,\\
    \end{aligned}
    \right.
    \end{gathered}
    \end{equation}
    so that $\mu_t^\star\in \al M_{+}$ for all $t\in(0,1)$ 
    and $\mu^\star_t$ is continuous at $t=\tfrac{1}{3},\tfrac{2}{3}$.
Define the source and potential by
    \begin{equation}
    \label{exm_ce_non_empty'}
    \begin{gathered}
    h_t^\star:=
    \left\{
    \begin{aligned}
    &3\big([\mu_1-\mu_0,\varpi]+\varepsilon\big), && 0\le t\le\tfrac{1}{3},\vspace{6pt}\\
    &0, && \tfrac{1}{3}< t\le\tfrac{2}{3},\vspace{6pt}\\
    &-3\varepsilon, && \tfrac{2}{3}< t\le 1,\\
    \end{aligned}
    \right.
    \\
    \psi_t^\star:=
    \left\{
    \begin{aligned}
    &0, && 0\le t\le\tfrac{1}{3},\vspace{6pt}\\
    &3\Big(B_{\mu_t^\star}\!\!\restriction_{\mathsf{H}}\Big)^{-1}
    \!\Big(
    (\mu_1-\mu_0)+[\mu_0-\mu_1,\varpi]\,p
    \Big),
    && \tfrac{1}{3}< t\le\tfrac{2}{3},\vspace{6pt}\\
    &0, && \tfrac{2}{3}< t\le 1,\\
    \end{aligned}
    \right.
    \end{gathered}
    \end{equation}
    where $(B_{\mu_t^\star}\!\!\restriction_{\mathsf{H}})^{-1}$ is well-defined according to Lemma \ref{lem_ce_solver}.
By construction, $(\mu_t^\star,\psi_t^\star,h_t^\star)_{t\in[0,1]}\in \mathrm{CE}_{p}(\mu_{0},\mu_{1};[0,1])$, with $(\mu^\star_t)_{t\in[0,1]}\in \mathrm{A\!C}([0,1];\al M)$ and $\mu_t^\star\in\al M_{+}$ for a.e.\ $t\in(0,1)$.
The parameter $\varepsilon$ guarantees strict positivity along the path. 
It can be chosen arbitrarily small.

\subsection{Metric axioms and completeness}
This part establishes the metric properties of $\al W^{a,b}_{p}$.
We begin with a linear–square (LinSq) representation of the action, 
    whose additivity under concatenation gives a clean proof of the triangle inequality.
We then show a local equivalence between $\al W^{a,b}_{p}$ and the $\mathrm{L}^{1}$-metric, 
    which yields finiteness and identity of indiscernibles and, in turn, 
    will be used to prove the completeness of $(\al M,\al W^{a,b}_{p})$.

For any time interval $[\tau,T]$,
    take $(\mu_t,\psi_t,h_t)_{t\in[\tau,T]}\in \mathrm{CE}_p(\mu_0,\mu_1;[\tau,T])$. 
Define 
    \begin{equation}
    \label{def_E_potential_linsq}
    E_{\rm LinSq}^{a,b}((\mu_t,\psi_t,h_t)_{t\in[\tau,T]})
    :=
    \left(
    \int_\tau^T\bigg\{
    a^2h^2_t+b^2[A_{\mu_t}\psi_t,\psi_t]
    \bigg\}^{\frac{1}{2}}\di t
    \right)^2.
    \end{equation}
Here the subscript ``LinSq'' indicates integrating the linear amplitude first and then squaring.

The next result is standard in the literature (e.g. \cite[Theorem 5.4]{dolbeault2009new}). 
For completeness, the proof is given in Appendix \ref{sec_app_a}.
\begin{lemma}
\label{lem_ELinSq_potential_rewrite}
For any $\mu_0,\mu_1\in \al M$,
\begin{equation*}
\al W_p^{a,b}(\mu_0,\mu_1)^2
= \inf 
\left\{
E_{\rm LinSq}^{a,b}((\mu_t,\psi_t,h_t)_{t\in[0,1]})
~\big|~
(\mu_t,\psi_t,h_t)_{t\in[0,1]}\in \mathrm{CE}_p(\mu_0,\mu_1;[0,1])
\right\}.
\end{equation*}
\end{lemma}

$E_{\rm LinSq}^{a,b}$ is invariant to the reparameterization of time.
Given intervals $I=[\tau,T]$ with $T>\tau$, write
    \begin{equation*}
    \mathrm{Path}(I):=\mathrm{A\!C}(I;\al M)\times \mathrm{L}^{1}(I;\bb R^{\al X})\times \mathrm{L}^{1}(I;\bb R).
    \end{equation*}
Let $\bar I=[\bar\tau,\bar T]$ with $\bar T>\bar\tau$. 
Define the affine time map
    \begin{equation*}
    \sigma_{I\to\bar I}(t):=\tau+\vartheta\,(t-\bar\tau),
    \qquad
    \vartheta(\tfrac{I}{\bar I}):=\frac{T-\tau}{\bar T-\bar\tau}.
    \end{equation*}
The time-reparameterization operator
    \begin{equation*}
    \mathscr{T}_{I\to\bar I}:\ \mathrm{Path}(I)\longrightarrow \mathrm{Path}(\bar I),
    \qquad
    \mathscr{T}_{I\to\bar I}\big((\mu_t,\psi_t,h_t)_{t\in I}\big)
    :=\big(\mu_{\sigma_{I\to\bar I}(t)},\ \vartheta(\tfrac{I}{\bar I})\,\psi_{\sigma_{I\to\bar I}(t)},\ \vartheta(\tfrac{I}{\bar I})\,h_{\sigma_{I\to\bar I}(t)}\big)_{t\in\bar I},
    \end{equation*}
    is a bijection with inverse $\mathscr{T}_{\bar I\to I}$.
Moreover, admissibility is preserved:
    if $(\mu_t,\psi_t,h_t)_{t\in I}\in\mathrm{CE}_p(\mu_0,\mu_1;I)$, then
    \begin{equation*}
    \mathscr{T}_{I\to\bar I}((\mu_t,\psi_t,h_t)_{t\in I})\in \mathrm{CE}_p(\mu_0,\mu_1;\bar I).
    \end{equation*}

\begin{lemma}
\label{lem_E_Linsq_tInvar}
For any $(\mu_t,\psi_t,h_t)_{t\in I}\in \mathrm{CE}_{p}(\mu_{0},\mu_{1};I)$,
    \begin{equation*}
    E_{\rm LinSq}^{a,b}((\mu_t,\psi_t,h_t)_{t\in I})
    =
    E_{\rm LinSq}^{a,b}
    \left(\mathscr{T}_{I\to\bar I}((\mu_t,\psi_t,h_t)_{t\in I})\right).
    \end{equation*}
Consequently, for any interval $I$,
    \begin{equation*}
    \al W_p^{a,b}(\mu_0,\mu_1)^2
    = \inf \left\{
    E_{\rm LinSq}^{a,b}((\mu_t,\psi_t,h_t)_{t\in I})
    ~\bigg|~(\mu_t,\psi_t,h_t)_{t\in I}\in \mathrm{CE}_p(\mu_0,\mu_1;I)
    \right\}.
    \end{equation*}
\end{lemma}
\begin{proof}
Without loss of generality, we consider for $T>0$ the operator $\mathscr{T}_{[0,1]\to[0,T]}$.
Let $m(s)=\frac{s}{T}$ for each $s\in[0,T]$
    and we have
    \begin{equation*}
    \mathscr{T}_{[0,1]\to[0,T]}((\mu_t,\psi_t,h_t)_{t\in[0,1]})
    =
    (\mu\circ m(s),{\textstyle\frac{1}{T}}\psi\circ m(s),{\textstyle\frac{1}{T}}h\circ m(s))_{s\in[0,T]}.
    \end{equation*}
A simple calculation yields that
    \begin{align*}
    \begin{aligned}
    &\left(E_{\rm LinSq}^{a,b}
    \left(\mathscr{T}_{[0,1]\to[0,T]}((\mu_t,\psi_t,h_t)_{t\in[0,1]})\right)\right)^{\frac{1}{2}}
   \\ &\qquad\qquad\qquad=
    \int_0^T
    \bigg\{
    a^2\big[{\textstyle\frac{1}{T}}h\circ m(s)\big]^2+b^2\big[A_{\mu\circ m(s)}\left({\textstyle\frac{1}{T}}\psi\circ m(s)\right),{\textstyle\frac{1}{T}}\psi\circ m(s)\big]
    \bigg\}^{\frac{1}{2}}
    \di s\\   
    &\qquad\qquad\qquad=
    \int_0^1\bigg\{
    a^2h^2_t+b^2[A_{\mu_t}\psi_t,\psi_t]
    \bigg\}^{\frac{1}{2}}\di t
    = 
    \left(E_{\rm LinSq}^{a,b}((\mu_t,\psi_t,h_t)_{t\in[0,1]})\right)^{\frac{1}{2}}.
    \end{aligned}
    \end{align*}
Combining this equality with Lemma \ref{lem_ELinSq_potential_rewrite} yields the desired rewrite.
\end{proof}

Now we show a local equivalence between $\al W^{a,b}_{p}$ and the $\mathrm{L}^{1}$-metric.
\begin{lemma}
\label{lem_l1_bound}
For $\mu_0,\mu_1\in \al M$, we have
    \begin{equation*}
    \begin{gathered}
    \mathsf C^{-1} {\al W}_p^{a,b}(\mu_0,\mu_1) 
    \le
    \| \mu_0-\mu_1\|_{\varpi,1}
    \le
    \mathsf c
    {\al W}_p^{a,b}(\mu_0,\mu_1),
    \end{gathered}
    \end{equation*}
    where the parameters $\mathsf c$ and $\mathsf C$ are defined as
    \begin{equation*}
    \begin{gathered}
    \mathsf c:= a^{-1}+b^{-1}\sqrt{\frac{2\,[\mu_0,\varpi]+2\,a^{-1}{\al W}_p^{a,b}(\mu_0,\mu_1)}{\min_{x\in\al X}\varpi(x)}}\,,
    \\
    \mathsf C:= \max
    \left\{
    a+
    b \Big(\big(\min_{x\in\al X} \varpi(x)\big)^{-1}+ | p |_2\Big)
    \| A_{\nu} \|^{\frac{1}{2}}
        \left\| 
        (B_{\nu}\!\!\restriction_{\mathsf{H}})^{-1}
        \right\|
        ~\bigg|~
        \nu\in \al M, \|\nu\|_{\varpi,1}\le \max_{i=1,2} \|\mu_i\|_{\varpi,1} 
    \right\}.
    \end{gathered}
    \end{equation*}
\end{lemma}
The standard proof is deferred to Appendix \ref{sec_app_c}.

\begin{remark}
Lemma \ref{lem_l1_bound} gives a local equivalence result between $\al W_p^{a,b}$ and the $\| \cdot \|_{\varpi,1}$-metric.
Yet the global equivalence fails to hold.
Below we show that
    there exists no constant $c_1>0$ such that
    \begin{equation*}
    \al W_p^{a,b}(\mu,\nu)\ge c_1 \|\mu-\nu\|_{\varpi,1}\qquad\text{for all }\mu,\nu\in \al M.
    \end{equation*}

Fix $\mu_0,\mu_1\in \al M_+$ with the same total mass. There exists a pure-transport curve in $\mathrm{CE}_p(\mu_0,\mu_1;[0,1])$. For example, set $\bar\mu_t:= (1-t)\mu_0+t\mu_1$ and choose $\bar\psi_t$ solving, for $t\in(0,1)$,
\begin{equation*}
\dot{\bar{\mu}}_t+\nabla\!\cdot\big(\hat{\bar\mu}_t\!\ast\!\nabla\bar\psi_t\big)=0.
\end{equation*}
By Lemma \ref{lem_ce_solver}, such $\bar\psi_t$ exists, hence $(\bar\mu_t,\bar\psi_t,0)_{t\in[0,1]}\in\mathrm{CE}_p(\mu_0,\mu_1;[0,1])$.
For any $\lambda>0$ define the scaled data
\begin{equation*}
\mu_0^{(\lambda)}:=\lambda\mu_0,\qquad
\mu_1^{(\lambda)}:=\lambda\mu_1,\qquad
\mu_t^{(\lambda)}:=\lambda\bar\mu_t,\qquad
\psi_t^{(\lambda)}:=\bar\psi_t.
\end{equation*}
Using the homogeneity $\theta(\lambda\alpha,\lambda\beta)=\lambda\theta(\alpha,\beta)$, we obtain
\begin{equation*}
\dot\mu_t^{(\lambda)}
+\nabla\!\cdot\big(\hat{\mu}^{(\lambda)}_t\!\ast\!\nabla\psi^{(\lambda)}_t\big)
=
\lambda\Big(\dot{\bar{\mu}}_t+\nabla\!\cdot\big(\hat{\bar\mu}_t\!\ast\!\nabla\bar\psi_t\big)\Big)=0,
\end{equation*}
so $(\mu^{(\lambda)},\psi^{(\lambda)},0)\in \mathrm{CE}_p(\mu_0^{(\lambda)},\mu_1^{(\lambda)};[0,1])$.

By the definition of $\al W_p^{a,b}$,
\begin{equation*}
\begin{aligned}
\al W_p^{a,b}\big(\mu_0^{(\lambda)},\mu_1^{(\lambda)}\big)^2
&\le
b^2\int_0^1
\frac{1}{2}\sum_{x,y\in \al X}\big(\psi_t^{(\lambda)}(y)-\psi_t^{(\lambda)}(x)\big)^2
\,\hat\mu_t^{(\lambda)}(x,y)K(x,y)\varpi(x)\,\di t
\\
&=
\lambda\,b^2\int_0^1
\frac{1}{2}\sum_{x,y\in \al X}\big(\bar\psi_t(y)-\bar\psi_t(x)\big)^2
\,\hat{\bar\mu}_t(x,y)K(x,y)\varpi(x)\,\di t.
\end{aligned}
\end{equation*}
Let
\begin{equation*}
C_*:=
b\left(\int_0^1
\frac{1}{2}\sum_{x,y\in \al X}\big(\bar\psi_t(y)-\bar\psi_t(x)\big)^2
\,\hat{\bar\mu}_t(x,y)K(x,y)\varpi(x)\,\di t\right)^{\frac{1}{2}}.
\end{equation*}
Then
\begin{equation*}
\al W_p^{a,b}\big(\mu_0^{(\lambda)},\mu_1^{(\lambda)}\big)\le C_*\sqrt{\lambda}.
\end{equation*}
On the other hand,
\begin{equation*}
\|\mu_0^{(\lambda)}-\mu_1^{(\lambda)}\|_{\varpi,1}
=\lambda\,\|\mu_0-\mu_1\|_{\varpi,1}.
\end{equation*}
Therefore,
\begin{equation*}
\frac{\al W_p^{a,b}\big(\mu_0^{(\lambda)},\mu_1^{(\lambda)}\big)}{\|\mu_0^{(\lambda)}-\mu_1^{(\lambda)}\|_{\varpi,1}}
\le
\frac{C_*}{\sqrt{\lambda}\,\|\mu_0-\mu_1\|_{\varpi,1}}
\longrightarrow 0
\quad\text{as }\lambda\to\infty,
\end{equation*}
which rules out the existence of a global constant $c_1>0$ with $\al W_p^{a,b}(\mu,\nu)\ge c_1\|\mu-\nu\|_{\varpi,1}$ for all $\mu,\nu\in \al M$. 
\end{remark}

Having set up the notation and basic tools, we now verify the distance axioms in the next theorem.
\begin{theorem}
\label{thm_metric}
$(\al M,\al W^{a,b}_p)$ is a metric space.
\end{theorem}

\vspace{0.1cm}
\begin{proof}~ We divide the proof into four parts to check all the distance axioms.
\vspace{0.2cm}

\noindent$\bullet$ \textbf{Part 1. Symmetry.} 
Here we show that trajectory-reversal is a one-to-one mapping between $\text{CE}_p(\mu_0,\mu_1;[0,1])$ and $\text{CE}_p(\mu_1,\mu_0;[0,1])$,
    and the reversal does not change the value of $E^{a,b}_{\rm Quad}$.

Let $(\mu_t,\psi_t,h_t)_{t\in[0,1]}$ be an arbitrary trajectory in $\mathrm{CE}_p(\mu_0,\mu_1;[0,1])$.
Define the reversed trajectory for $t\in[0,1]$,
\begin{equation*}
    \reflectbox{$\mu$}_t:=\mu_{1-t},
    \quad
    \reflectbox{$\psi$}_t:=-\psi_{1-t},
    \quad
    \reflectbox{$h$}_t:=-h_{1-t}.
\end{equation*}
Then they satisfy the continuity equation:
\begin{equation*}
    \dot{\reflectbox{$\mu$}}_t
    =
    -\dot \mu_{1-t}
    =
    -B_{\mu_{1-t}}\psi_{1-t}-h_{1-t}p
    =
    B_{\reflectbox{$\mu$}_t}\reflectbox{$\psi$}_t+\reflectbox{$h$}_t p.
\end{equation*}
The energy along the reversed trajectory remains unchanged:
\begin{equation*}
\begin{aligned}
    E_{\rm Quad}^{a,b}((\reflectbox{$\mu$}_t,\reflectbox{$\psi$}_t,\reflectbox{$h$}_t)_{t\in[0,1]})
    =&
    a^2 \int_0^1 \reflectbox{$h$}_t^2 \di t
    +
    b^2 \int_0^1 \| \nabla \reflectbox{$\psi$}_t \|_{\reflectbox{$\mu$}_t}^2 \di t
    \\
    =&
    a^2 \int_0^1 h_{1-t}^2 \di t
    +
    b^2 \int_0^1 \| \nabla \psi_{1-t} \|_{\mu_{1-t}}^2 \di t
    \\
    =&
    E_{\rm Quad}^{a,b}((\mu_t,\psi_t,h_t)_{t\in[0,1]}).
\end{aligned}
\end{equation*}\vspace{0.2cm}

\noindent$\bullet$ \textbf{Part 2. Triangle inequality.} 
To prove the triangle inequality, we concatenate admissible trajectories.
Fix $\mu_2\in\al M$, and take
    \begin{equation*}
    (\mu_t^{02},\psi_t^{02},h_t^{02})_{t\in[0,1]}\in \mathrm{CE}_p(\mu_0,\mu_2;[0,1]),
    \quad
    (\mu_t^{21},\psi_t^{21},h_t^{21})_{t\in[0,1]}\in \mathrm{CE}_p(\mu_2,\mu_1;[0,1]),
    \end{equation*}
    where the superscript $\Box^{02}$ labels a curve from $\mu_0$ to $\mu_2$, $\Box^{21}$ a curve from $\mu_2$ to $\mu_1$.
Define the glued trajectory on $[0,2]$ that starts at $\mu_0$, passes through $\mu_2$ at time $1$, and ends at $\mu_1$ by
    \begin{equation*}
    (\mu_t^{021},\psi_t^{021},h_t^{021})
    =
    \left\{
    \begin{aligned}
    &(\mu_t^{02},\psi_t^{02},h_t^{02}),&& \text{if} \ t\in[0,1],\\
    &(\mu_{t-1}^{21},\psi_{t-1}^{21},h_{t-1}^{21}),&& \text{if} \ t\in(1,2].
    \end{aligned}
    \right.
    \end{equation*}
Then $(\mu_t^{021},\psi_t^{021},h_t^{021})_{t\in[0,2]}\in \mathrm{CE}_p(\mu_0,\mu_1;[0,2])$.
By Lemma~\ref{lem_E_Linsq_tInvar} and additivity of the integral,
    \begin{equation*}
    \begin{aligned}
    \al W_p^{a,b}(\mu_0,\mu_1)
    \le & \
    (E_{\rm LinSq}^{a,b}((\mu_t^{021},\psi_t^{021},h_t^{021})_{t\in[0,2]}))^{\frac{1}{2}}
    \\
    = & \
    \int_0^2\bigg\{
    a^2(h^{021}_t)^2+b^2[A_{\mu_t^{021}}\psi_t^{021},\psi_t^{021}]
    \bigg\}^{\frac{1}{2}}\di t
    \\
    = & \
    \int_0^1\bigg\{
    a^2(h^{02}_t)^2+b^2[A_{\mu_t^{02}}\psi_t^{02},\psi_t^{02}]
    \bigg\}^{\frac{1}{2}}\di t
    +
    \int_0^1\bigg\{
    a^2(h^{21}_t)^2+b^2[A_{\mu_t^{21}}\psi_t^{21},\psi_t^{21}]
    \bigg\}^{\frac{1}{2}}\di t
    \\
    = & \
    (E_{\rm LinSq}^{a,b}((\mu_t^{02},\psi_t^{02},h_t^{02})_{t\in[0,1]}))^{\frac{1}{2}}
    +
    (E_{\rm LinSq}^{a,b}((\mu_t^{21},\psi_t^{21},h_t^{21})_{t\in[0,1]}))^{\frac{1}{2}}.
    \end{aligned}
    \end{equation*}
Taking the infimum on the right-hand side over all admissible $(\mu_t^{02},\psi_t^{02},h_t^{02})_{t\in[0,1]}$ and $(\mu_t^{21},\psi_t^{21},h_t^{21})_{t\in[0,1]}$ yields
    \begin{equation*}
    \al W_p^{a,b}(\mu_0,\mu_1) \le \al W_p^{a,b}(\mu_0,\mu_2)+ \al W_p^{a,b}(\mu_2,\mu_1).
    \end{equation*}
    
\vspace{0.2cm}

\noindent$\bullet$ \textbf{Part 3. Non-negativity and finiteness.} 
The non-negativity is immediate in view of the definition \eqref{eq_def_w}.
For finiteness, in light of Lemma \ref{lem_l1_bound}, we have for any $\mu_0,\mu_1\in \al M$,
    \begin{equation*}
    \al W_p^{a,b}(\mu_0,\mu_1)
    \le \mathsf{C} \| \mu_0-\mu_1\|_{\varpi,1}<\infty,
    \end{equation*}
    where $\mathsf{C}$ is a positive parameter depending on $\mu_0,\mu_1$.
\vspace{0.2cm}
            
\noindent$\bullet$ \textbf{Part 4. Identity of indiscernibles.}
If $\mu_0=\mu_1$, then the constant curve $(\mu_0,0,0)_{t\in[0,1]}$ is admissible, which yields $\al W_p^{a,b}(\mu_0,\mu_1)\le 0$.
Thus ${\al W}_p^{a,b}(\mu_0,\mu_1)=0$.

Meanwhile, if ${\al W}_p^{a,b}(\mu_0,\mu_1)=0$, 
    then Lemma \ref{lem_l1_bound} guarantees the existence of $\mathsf{c}\in(0,\infty)$ such that
    \begin{equation*}
    \| \mu_0-\mu_1\|_{\varpi,1}\le \mathsf{c} \al W_p^{a,b}(\mu_0,\mu_1)=0.
    \end{equation*}
Therefore, $\mu_0=\mu_1$.
\end{proof}

We have now established that ${\al W}_p^{a,b}$ is a metric.
We now prove that $(\mathcal M,\mathcal W_p^{a,b})$ is complete, using Lemma \ref{lem_l1_bound}.
\begin{corollary}
\label{cor_complete}
$(\al M,{\al W}_p^{a,b})$ is complete.
\end{corollary}
\begin{proof}
Assume that $\{\mu_n\}_{n\in \mathbb{N}_+}$ is a Cauchy sequence with respect to the metric $\al W_p^{a,b}$.
Then, by Lemma \ref{lem_l1_bound}, for any $n,m\in\mathbb{N}_+$,
    \begin{equation*}
    \| \mu_n-\mu_m\|_{\varpi,1} 
    \le \mathsf{c}_{n,m}
    {\al W}_p^{a,b}(\mu_n,\mu_m),
    \end{equation*}
    where
    \begin{equation*}
    \mathsf{c}_{n,m} =
    a^{-1}+b^{-1}\sqrt{\frac{2\,[\mu_n,\varpi]+2\,a^{-1}{\al W}_p^{a,b}(\mu_n,\mu_m)}{\min\!\varpi}}.
    \end{equation*}
By the Cauchy property, $\{\mu_n\}_{n\in \mathbb{N}_+}$ is $\al W_p^{a,b}$ bounded.
In particular, there exists $M>0$ such that
    \begin{equation*}
    \al W_p^{a,b}(\mu_1,\mu_n) \le M, 
    \quad \al W_p^{a,b}(\mu_n,\mu_m) \le M,
    \quad \forall~ n,m\in\mathbb{N}_+.
    \end{equation*}
For each $\ep>0$, 
    let $(\mu^{n,\ep}_t,\psi^{n,\ep}_t,h^{n,\ep}_t)_{t\in[0,1]}\in \mathrm{CE}_p(\mu_1,\mu_n;[0,1])$ such that
    \begin{equation*}
    \big(E_{\rm Quad}^{a,b}((\mu^{n,\ep}_t,\psi^{n,\ep}_t,h^{n,\ep}_t)_{t\in[0,1]}) \big)^{\frac{1}{2}}
    \le \al W_p^{a,b}(\mu_1,\mu_n) +\ep.
    \end{equation*}
Then
    \begin{equation*}
    \begin{aligned}
    \relax
    [\mu_n,\varpi] 
    \le&
    [\mu_1,\varpi] 
    + \int_0^1 \left| \frac{\di}{\di t} [\mu_t^{n,\ep},\varpi]\right| \di t
    =
    [\mu_1,\varpi] 
    + \int_0^1 \left| h_t^{n,\ep}\right| \di t
    \\
    \le &
    [\mu_1,\varpi] 
    + a^{-1} \big(E_{\rm Quad}^{a,b}((\mu^{n,\ep}_t,\psi^{n,\ep}_t,h^{n,\ep}_t)_{t\in[0,1]}) \big)^{\frac{1}{2}}
    \\
    \le &
    [\mu_1,\varpi] + \frac{M+\ep}{a}.
    \end{aligned}
    \end{equation*}
By the arbitrariness of $\ep$, we have
    \begin{equation*}
    [\mu_n,\varpi] \le [\mu_1,\varpi] +\frac{M}{a} .
    \end{equation*}
Thus $\mathsf c_{n,m}$ has a uniform upper bound:
    \begin{equation*}
    \mathsf c_{n,m} \le \bar c := 
    a^{-1}+b^{-1}
    \sqrt \frac{2a[\mu_1,\varpi]+4M}{a \min\! \varpi}.
    \end{equation*}
Therefore, $\{\mu_n\}$ is Cauchy with respect to $\| \cdot\|_{\varpi,1}$-norm \eqref{L_1norm} and there exists $\mu_*\in\al M$ such that
    \begin{equation*}
    \lim_{n\to\infty}\| \mu_n-\mu_* \|_{\varpi,1} =0.
    \end{equation*}
By Lemma \ref{lem_l1_bound}, we also have
    \begin{equation*}
    \al W_p^{a,b}(\mu_n,\mu_*)
    \le 
    \mathsf{C}_n \| \mu_n-\mu_*\|_{\varpi,1},
    \end{equation*}
    where
    \begin{equation*}
    \mathsf C_n=\max_{[\nu,\varpi]\le \max\{[\mu_n,\varpi],[\mu_\ast,\varpi]\}}
    \left\{
    a+
    b \Big(\min\!\varpi^{-1}+ \| p \|_2\Big)
    \| A_{\nu} \|^{\frac{1}{2}}
        \left\| 
        (B_{\nu}\!\!\restriction_{\mathsf{H}})^{-1}
        \right\| 
    \right\}.
    \end{equation*}
As we already have an upper bound for $[\mu_n,\varpi]$,
    let
    \begin{equation*}
    \bar {\mathsf  C} :=\max_{[\nu,\varpi]\le [\mu_1,\varpi]+[\mu_\ast,\varpi]+a^{-1}M}
    \left\{
    a+
    b \Big(\min\!\varpi^{-1}+ \| p \|_2\Big)
    \| A_{\nu} \|^{\frac{1}{2}}
        \left\| 
        (B_{\nu}\!\!\restriction_{\mathsf{H}})^{-1}
        \right\| 
    \right\},
    \end{equation*}
    then
    \begin{equation*}
    \lim_{n\to \infty}\al W_p^{a,b}(\mu_n,\mu_*)
    \le 
    \bar {\mathsf  C} \lim_{n\to \infty}\| \mu_n-\mu_*\|_{\varpi,1}
    =0.
    \end{equation*}
\end{proof}

\section{Gradient flows of the entropy functional}
\label{sec_flow}
In this section, we establish the Riemannian structures with our Benamou–Brenier metric and show that the heat equation is the gradient flow of the entropy functional.
	
\subsection{Riemannian structures}

As in classical optimal transport, the tangent representation is determined by the continuity equation and the Riemannian inner product is read off from the action density.
Recall the continuity equation
\begin{equation*}
  \dot{\mu}_t+\nabla_{\!\mu_t}\!\cdot\nabla\psi_t=h_t\,p,
\end{equation*}
and the instantaneous action
\begin{equation*}
  a^2 h_t^2+b^2\,\langle\nabla\psi_t,\nabla\psi_t\rangle_{\mu_t}.
\end{equation*}
Since $\al M_+\subset\mathbb{R}^{\al X}$ is open, we canonically identify $T_\mu \al M_+=\mathbb{R}^{\al X}$.

Given an arbitrary $\rho\in T_\mu \al M_+$,
    the unique representation $(\nabla\psi_\rho,h_\rho)$ solving
    \begin{equation*}
    \rho+\nabla_{\!\mu}\!\cdot\nabla\psi_\rho=h_\rho\,p
    \end{equation*}
    is given by Lemma \ref{lem_source_ce_solver}.
And recall that we write the decomposition as in \eqref{eq:Tmu-decomposition}:
    \begin{equation*}
    D_\mu(\rho):=(\nabla\psi_\rho,h_\rho)
    :=\big(J_\mu([\rho,\varpi]\,p-\rho),\,[\rho,\varpi]\big),
    \end{equation*}
    where $J_\mu$ is defined by \eqref{eq_Jmu_characterization}:
    \begin{equation*}
    J_\mu:\nu\in \mathsf{H}_\varpi
    \mapsto
    \nabla\big((B_\mu\!\restriction_{\mathsf{H}})^{-1}\nu\big)
    \in
    \{\Psi\in\mathbb{R}^{\al X\times \al X} \mid \ \exists\,\psi\in\mathbb{R}^{\al X},\ \Psi=\nabla\psi\}.
    \end{equation*}
Finally, for $\mu\in \al M_+$ and $\rho,\xi\in T_\mu \al M_+$, with
$D_\mu(\rho)=(\nabla\psi_\rho,h_\rho)$ and $D_\mu(\xi)=(\nabla\psi_\xi,h_\xi)$, define
    \begin{equation}
    \label{eq:g-mu-def}
    \mathbf{g}_\mu(\rho,\xi):=a^2 h_\rho h_\xi+b^2\,\langle\nabla\psi_\rho,\nabla\psi_\xi\rangle_\mu,
    \end{equation}
    equivalently,
    \begin{equation}
    \label{eq:g-mu-operator}
    \mathbf{g}_\mu(\rho,\xi)
    =a^2\,[\rho,\varpi]\,[\xi,\varpi]
    +b^2\,\big\langle J_\mu([\rho,\varpi]\,p-\rho),\,J_\mu([\xi,\varpi]\,p-\xi)\big\rangle_\mu.
    \end{equation}

For the identification length distance with $\al W^{a,b}_p$, we require the following lemma.
It allows us to restrict the dynamic minimization to trajectories that remain in $\al M_+$ on $(0,1)$ without changing the value of the infimum, 
    so that the tangent decomposition and the $\mathbf{g}$–metric are well defined along the whole path.
\begin{lemma}
\label{lem_positive_admissible_set}
For every $\mu_0,\mu_1\in \al M $, we have
    \begin{align}
    \label{eq_positive_admissible_set_psi}
    \begin{aligned}
    \al W^{a,b}_p(\mu_0,\mu_1)
    = &
    \inf\Big\{
    {E}^{a,b}_{\rm LinSq}(\mu_t,\psi_t,h_t)_{t\in[0,1]}^{\frac{1}{2}}
    \mid
    (\mu_t,\psi_t,h_t)_{t\in[0,1]}\in \mathrm{CE}_p(\mu_0,\mu_1;[0,1])\text{ and }
    \mu_t:(0,1)\rightarrow \al M_+
    \Big\}.  
    \end{aligned}
    \end{align}
\end{lemma}
The proof is an adaptation of \cite[Lemma~3.30]{maas2011gradient}. 
For completeness, we defer it to Appendix \ref{sec_app_g}.

\begin{theorem}
\label{thm:g-is-Riemannian}
The bilinear form $\mathbf{g}_\mu$ of \eqref{eq:g-mu-def} defines a smooth Riemannian metric on $\al M_+$.
Moreover, the induced length distance $\mathrm{dist}_{\mathbf{g}}$ coincides with the dynamic distance $\al W^{a,b}_p$.
\end{theorem}

\begin{proof}
We split the proof into four steps.
\vspace{0.2cm}

\noindent $\bullet$ \textbf{Step 1. Bilinearity and symmetry.}
By the definition of $D_\mu$ in \eqref{eq:Tmu-decomposition} 
    and the linearity of $J_\mu$ in \eqref{eq_Jmu_characterization}, 
    the map $\rho\mapsto D_\mu(\rho)$ is linear.
Hence $\mathbf{g}_\mu$ is bilinear and symmetric by \eqref{eq:g-mu-def}.

\medskip
\noindent $\bullet$ \textbf{Step 2. Positive definiteness.}
Let $\rho\in T_\mu \al M_+$ and 
    write $(\nabla\psi_\rho,h_\rho):=D_\mu(\rho)$.
If $\mathbf{g}_\mu(\rho,\rho)=0$, 
    then $a^2h_\rho^2=0$ so $h_\rho=0$, 
    and $b^2\langle\nabla\psi_\rho,\nabla\psi_\rho\rangle_\mu=0$.
By the discrete inner product,
    \begin{equation*}
    \langle\nabla\psi_\rho,\nabla\psi_\rho\rangle_\mu
    = \frac{1}{2}\sum_{x,y\in \al X}(\psi_\rho(y)-\psi_\rho(x))^2\,\hat\mu(x,y)\,K(x,y)\,\varpi(x).
    \end{equation*}
Since $\hat\mu$ and $\varpi$ are strictly positive on $\al M_+$, 
    the vanishing of the sum implies $\psi_\rho(y)=\psi_\rho(x)$ whenever $K(x,y)>0$.
By connectedness, $\psi_\rho$ is constant and therefore $\nabla\psi_\rho\equiv 0$.
Returning to \eqref{eq_tangent_identification_0} with $h_\rho=0$ gives 
    $\rho=-\nabla_\mu\cdot \nabla\psi_\rho+h_\rho p=0$.
Thus $\mathbf{g}_\mu$ is positive definite.

\medskip
\noindent $\bullet$ \textbf{Step 3. Smooth dependence on $\mu$.}
On $\al M_+$ the weight $\theta$ is smooth, 
    hence $\mu\mapsto B_\mu$ is a smooth matrix field.
By Lemma~\ref{lem_B_mu_isormorphism}, 
    $B_\mu\!\restriction_{\mathsf{H}}$ is an isomorphism 
    between $\mathsf{H}$ and $\mathsf{H}_{\varpi}$
    for every $\mu\in \al M_+$, 
    and its inverse depends smoothly on $\mu$.
Therefore $\mu\mapsto J_\mu$ is smooth by \eqref{eq_Jmu_characterization}, 
    and consequently $\mu\mapsto \mathbf{g}_\mu$ is smooth by \eqref{eq:g-mu-operator}.

\medskip
\noindent $\bullet$ \textbf{Step 4. Identification of the length with the dynamic action.}
Let $\mu_t:[0,1]\to \al M_+$ be absolutely continuous.
Using \eqref{eq:Tmu-decomposition}, 
    define $(\nabla\psi_t,h_t):=D_{\mu_t}(\dot\mu_t)$.
Then the $\mathbf{g}$-norm of the velocity reads
    \begin{equation*}
    \|\dot\mu_t\|_{\mathbf{g}}^2 \;=\; a^2\,h_t^2 \,+\, b^2\,\langle\nabla\psi_t,\nabla\psi_t \rangle_{\mu_t},
    \end{equation*}
    so the $\mathbf{g}$-length of $(\mu_t)_{t\in[0,1]}$ is
    \begin{equation*}
    L_{\mathbf{g}}((\mu_t)_{t\in[0,1]})\ =\ \int_0^1 \sqrt{a^2 h_t^2 + b^2\langle\nabla\psi_t,\nabla\psi_t \rangle_{\mu_t}}\di t.
    \end{equation*}
Conversely, given any feasible triple 
    $(\mu_t,\psi_t,h_t)_{t\in[0,1]}\in\mathrm{CE}_p(\mu_0,\mu_1;[0,1])$ 
    with $\mu_t|_{[0,1]}\in \al M_+$, 
    the above identity shows that the action 
    $(E^{a,b}_{\rm LinSq}((\mu_t,\psi_t,h_t)_{t\in[0,1]}))^{\frac{1}{2}}$ given in \eqref{def_E_potential_linsq}
    is exactly the $\mathbf{g}$-length of $\mu$.
Therefore:
    \begin{itemize}
    \item For any admissible dynamic triple $(\mu_t,\psi_t,h_t)_{t\in[0,1]}$ joining $(\mu_0,\mu_1)$ with $\mu_t|_{[0,1]}\in \al M_+$, 
        the dynamic length
        $(E^{a,b}_{\rm LinSq}((\mu_t,\psi_t,h_t)_{t\in[0,1]}))^{\frac{1}{2}}=\int_0^1 \sqrt{a^2 h_t^2 + b^2[A_{\mu_t}\psi_t,\psi_t]}\di t$ 
        equals $L_{\mathbf{g}}(\mu)$.
    By Lemma \ref{lem_positive_admissible_set}, 
        taking the infimum over all such triples gives $\al W^{a,b}_p(\mu_0,\mu_1)\ge \mathrm{dist}_{\mathbf{g}}(\mu_0,\mu_1)$.
    \item For any absolutely continuous curve $\mu$ with 
        $\mu(0)=\mu_0$, $\mu(1)=\mu_1$,
        the decomposition \eqref{eq:Tmu-decomposition} produces a feasible dynamic triple, 
        hence $L_{\mathbf{g}}(\mu)$ is an admissible dynamic length, 
        so $\mathrm{dist}_{\mathbf{g}}(\mu_0,\mu_1)\ge \al W^{a,b}_p(\mu_0,\mu_1)$.
    \end{itemize}
Combining the two inequalities yields $\mathrm{dist}_{\mathbf{g}}=\al W^{a,b}_p$.
\end{proof}

\subsection{Entropy gradient flows and long-time asymptotics}
It is well known that, in Euclidean settings, the heat flow is the gradient flow of entropy with respect to the Wasserstein metric.
Here, in the $\mathbf{g}$–geometry, we shall identify the gradient flow of entropy with a heat equation modified by an external source, 
    and we describe its long–time behavior via a Łojasiewicz–type inequality.

Let the discrete entropy functional be defined for $\mu\in \al M$ by
    \begin{equation}
    \label{eq_H_def}
    \bb H(\mu)
    := \sum_{x\in \al X}\big(\mu(x)\log\mu(x)-\mu(x)\big)\,\varpi(x)
    = \langle \mu,\,\log\mu-1\rangle_{\varpi},
    \end{equation}
    with the boundary convention that $0\times\log 0=0$.
For mass–preserving geometries on probability measures (e.g., \cite{maas2011gradient}), 
    it is immaterial to replace $\bb H$ by
    \begin{equation*}
    \mathcal H(\mu):=\sum_{x\in \al X}\mu(x)\log\mu(x)\,\varpi(x)=\langle \mu,\,\log\mu\rangle_{\varpi},
    \end{equation*}
because $\sum_x\mu(x)\varpi(x)=1$ for all probability densities.
In our nonconservative setting, however, $\sum_x\mu(x)\varpi(x)$ is not fixed along admissible curves in $\al M_+$, 
    so the linear term in \eqref{eq_H_def} must be retained.

The next result illustrates this distinction:
    the entropy gradient flow comprises a source term in addition to the Laplacian dissipation.
Recall that for $\psi\in\mathbb{R}^{\al X}$ we set $\Delta\psi:=\nabla\cdot\nabla\psi$.

\begin{proposition}
\label{prop_gradient_flow}
Consider the generalized heat equation
\begin{equation}
  \label{eq_heat}
  \dot{\varrho}_t
  = b^{-2}\Delta\varrho_t
  - a^{-2}\,\langle \log\varrho_t,\,p\rangle_{\varpi}\,p,
\end{equation}
posed on $\al M_+$.
The semigroup induced by \eqref{eq_heat} is the gradient flow of $\bb H$ with respect to $\mathbf{g}$.
\end{proposition}
\begin{proof}
For $\mu\in \al M_+$ and $x,y\in \al X$,
    \begin{equation*}
    \nabla\mu(x,y)
    = \mu(y)-\mu(x)
    = \frac{\mu(y)-\mu(x)}{\log\mu(y)-\log\mu(x)}\big(\log\mu(y)-\log\mu(x)\big)
    = \hat\mu(x,y)\,(\nabla\log\mu)(x,y).
    \end{equation*}
Here we use the diagonal convention $(u-u)/(\log u-\log u)=u$ ($u>0$), so the identity also holds when $\mu(x)=\mu(y)$.
Hence
\begin{equation*}
  \Delta\mu \;=\; \nabla\!\cdot\!\big(\hat\mu \ast \nabla\log\mu\big).
\end{equation*}
Applying this to $\varrho_t$ and using \eqref{eq_heat} gives
\begin{equation*}
  \dot\varrho_t
  \;=\;
  \nabla\!\cdot\!\big(\hat\varrho_t \ast b^{-2}\nabla\log\varrho_t\big)
  \;-\;
  a^{-2}\,\langle \log\varrho_t,p\rangle_{\varpi}\,p,
\end{equation*}
which, compared with the continuity form
$\dot\mu_t+\nabla\!\cdot(\hat\mu_t\ast\nabla\psi_t)=h_t\,p$,
yields the tangent decomposition
\begin{equation*}
  D_{\varrho_t}(\dot\varrho_t)
  \;=\;
  \big(-\,b^{-2}\nabla\log\varrho_t,\,-\,a^{-2}\langle \log\varrho_t,p\rangle_{\varpi}\big).
\end{equation*}

Meanwhile, let $\mu_t:[0,1]\to \al M_+$ be absolutely continuous and write
$(\nabla\psi_t,h_t):=D_{\mu_t}(\dot\mu_t)$.
Then
\begin{equation*}
  \frac{\di}{\di t}\,\bb H(\mu_t)
  = \sum_{x\in \al X}\log\mu_t(x)\,\dot\mu_t(x)\,\varpi(x)
  = \big\langle \log\mu_t,\,-\,\nabla\!\cdot(\hat\mu_t\ast\nabla\psi_t)+h_t\,p\big\rangle_{\varpi}.
\end{equation*}
By discrete integration by parts \eqref{eq_int_by_parts},
    \begin{equation*}
    \frac{\di}{\di t}\,\bb H(\mu_t)
    = \big\langle \nabla\log\mu_t,\;\hat\mu_t\ast\nabla\psi_t\big\rangle_{\varpi}
    + \big\langle \log\mu_t, p\big\rangle_{\varpi}\,h_t
    = b^2\big\langle b^{-2}\nabla\log\mu_t,\nabla\psi_t\big\rangle_{\mu_t}
    + a^2\big(a^{-2}\langle \log\mu_t,p\rangle_{\varpi}\big)h_t.
    \end{equation*}
By the definition of the metric pairing $\mathbf{g}_{\mu_t}$ on tangents,
    \begin{equation*}
    \frac{\di}{\di t}\,\bb H(\mu_t)
    = \mathbf{g}_{\mu_t}\!\left( (\nabla\psi_t,h_t),\;\big(b^{-2}\nabla\log\mu_t,\;a^{-2}\langle \log\mu_t,p\rangle_{\varpi}\big) \right).
    \end{equation*}
Therefore, the $\mathbf{g}$–gradient of $\bb H$ at arbitrary $\mu\in\al M_+$ is represented by
\begin{equation*}
  D_{\mu}\,(\mathrm{grad}\,\bb H(\mu))
  = \big(b^{-2}\nabla\log\mu,\;a^{-2}\langle \log\mu,p\rangle_{\varpi}\big).
\end{equation*}
Evaluating at $\mu=\varrho_t$ and comparing with the tangent of the flow computed above,
\begin{equation*}
  D_{\varrho_t}(\mathrm{grad}\,\bb H(\varrho_t))
  = -\,D_{\varrho_t}(\dot\varrho_t),
\end{equation*}
hence $\mathrm{grad}\,\bb H(\varrho_t)=-\,\dot\varrho_t$, i.e., the semigroup of \eqref{eq_heat} is the gradient flow of $\bb H$ with respect to $\mathbf{g}$.
\end{proof}

Having identified the evolution \eqref{eq_heat} as the $\mathbf{g}$–gradient flow of the entropy \eqref{eq_H_def}, 
    we aim to record the well–posedness of the associated Cauchy problem,
    which requires the following variant of Danskin's theorem. 
A detailed proof can be found in Appendix \ref{sec_app_h}.
\begin{lemma}
\label{lem_differentiation_maximum}
Let $J:=[j_0,j_1)\subset[0,\infty)$.
Assume that for each $x\in \al X$ the map $t\mapsto \eta_t(x)$ is $\mathrm C^{1}$ on $J$.
Define
    \begin{equation*}
    \zeta(t):=\min_{x\in \al X}\eta_t(x),
    \qquad
    I(t):=\{x\in \al X \mid \eta_t(x)=\zeta(t)\},\qquad t\in J.
    \end{equation*}
Then the right derivative $\zeta'_+$ exists on $J$.
Moreover, for all $t\in J$,
    \begin{equation*}
    \zeta'_+(t)\;=\;\min_{x_t\in I(t)}\dot\eta_t(x_t).
    \end{equation*}
\end{lemma}

The following lemma ensures that the generalized heat dynamics defines a global semigroup on $\al M_+$, 
    which will be the starting point for the analysis of long–time behavior.
\begin{lemma}
\label{lem_heat_eq_wellposed}
For any initial state $\varrho_0\in \al M_+$, 
    the Cauchy problem for \eqref{eq_heat}
    admits a unique global solution $\varrho\in \mathrm C^1([0,\infty);\al M_+)$.
Moreover, the trajectory is contained in a compact subset of $\al M_+$.
\end{lemma}
\begin{proof}
The proof proceeds in three steps. 
First, local existence and uniqueness follow from the Picard–Lindelöf theorem applied to the smooth vector field on the open cone $\al M_+$. 
Second, entropy dissipation yields a priori bounds along solutions, which preclude finite–time blow-up.
Third, exploiting the behavior of $\log$ near zero, 
    we show that the trajectory has a uniform lower bound,
    thus we find the compact invariant set 
    and extend the local solution to a global one.

\medskip
\noindent$\bullet$ \textbf{Step 1. Local well-posedness.}
The vector field 
    $\varrho\mapsto b^{-2}\Delta\varrho-a^{-2}\langle \log\varrho,p\rangle_\varpi\,p$ 
    is $\mathrm C^\infty$ on the open set $\al M_+$, 
    hence by the Picard–Lindel\"of theorem 
    there exist $T_{\max}\in(0,\infty]$ and 
    a unique maximal solution $\varrho:[0,T_{\max})\to \al M_+$ with $\varrho(0)=\varrho_0$.

\medskip

\noindent$\bullet$ \textbf{Step 2. Entropy dissipation and a priori upper bound.}
By Proposition \ref{prop_gradient_flow},
    the flow \eqref{eq_heat} is the gradient flow of $\mathbb{H}$ and we have
    \begin{equation*}
    \frac{\di}{\di t} \mathbb{H}(\varrho_t)
    =
    - \mathbf{g}_{\varrho_t} (\mathrm{grad}\,\bb H(\varrho_t),\mathrm{grad}\,\bb H(\varrho_t) )
    \le 0.
    \end{equation*}
Thus $\mathbb H(\varrho_t)\le \mathbb H(\varrho_0)$ for all $t<T_{\max}$.
Since the function $r\mapsto r(\log r-1)$ is coercive on $(0,\infty)$, the sublevel set
    \begin{equation*}
    \{\,u\in \al M \mid\ \mathbb H(u)\le \mathbb H(\varrho_0)\,\}
    \end{equation*}
    is bounded.
Therefore $\sup_{t<T_{\max}} \max_{y\in \al X} |\varrho_t(y)|<\infty$.

\medskip
\noindent$\bullet$ \textbf{Step 3. Positive lower bound.}
For $t\in[0,T_{\max})$, let
    \begin{equation*}
    \phi(t):=\min_{x\in \al X}\varrho_t(x), \qquad
    I(t):=\operatorname{Argmin}_{x\in \al X}\varrho_t(x).
    \end{equation*}

According to Lemma \ref{lem_differentiation_maximum}, 
    the right derivative $\phi'_+$ exists across $[ 0,T_{\max} )$ and we have 
    \begin{equation}
    \label{eq_heat_wellposed_1}
    \phi'_+(t) = \min_{x_t\in I(t)} \dot\varrho_t(x_t)
    =
    \min_{x_t\in I(t)}
    \left\{
        b^{-2}(\Delta\varrho_t)(x_t)-
        a^{-2}\langle \log\varrho_t,p\rangle_\varpi p(x_t)
    \right\}.
    \end{equation}
By definition of $I(t)$ and the identity \eqref{eq_Delta=K-I}, for a fixed $x_t\in I(t)$, 
    $(\Delta\varrho_t)(x_t)=\sum_{y\in \al X} K(x_t,y)\big(\varrho_t(y)-\varrho_t(x_t)\big)\ge 0$.
For the other term in \eqref{eq_heat_wellposed_1}, we decompose
    \begin{equation*}
    \langle \log\varrho_t,p\rangle_\varpi=\varpi(x_t)p(x_t)\log\varrho_t(x_t)+R(t), 
    \qquad 
    R(t):=\sum_{y\neq x_t}\varpi(y)p(y)\log\varrho_t(y).
    \end{equation*}
In step 2, we have established that 
    $\sup_{t\in[0,T_{\max})} \max_{y\in \al X} |\varrho_t(y)|<\infty$,
    so the quantity $R(t)$ is bounded from above:
    \begin{equation*}
    \sup_{t\in[0,T_{\max})}R(t)
    \le 
    \sum_{y\in \al X} (\varpi(y) p(y) )
    \max\Big\{
    0,
    \log
    \Big(
        \sup_{t\in[0,T_{\max})}
        \max_{y\in \al X} |\varrho_t(y)|
    \Big)
    \Big\}
    =:C\in[0,\infty).
    \end{equation*}
Therefore, coming back to \eqref{eq_heat_wellposed_1},
    \begin{equation}
    \label{eq_heat_wellposed_2}
    \begin{aligned}
    \phi'_+(t) = 
    \min_{x_t\in I(t)}
    \dot\varrho_t(x_t)
    =&
    \min_{x_t\in I(t)}
    \left\{
        b^{-2}(\Delta\varrho_t)(x_t)-a^{-2} \langle \log\varrho_t,p\rangle_\varpi  p(x_t)
    \right\}
    \\
    \ge& 
    \min_{x_t\in I(t)}
    \left\{
        -a^{-2}\varpi(x_t)p(x_t)^2\log\varrho_t(x_t)-a^{-2}R(t)\,p(x_t)
    \right\}
    \\
    \ge& 
    \min_{x_t\in I(t)}
    \left\{
        -a^{-2}\varpi(x_t)p(x_t)^2 \sgn(\log \phi(t))
    \right\}
    |\log \phi(t) |
    -a^{-2}C\max_{y\in \al X}p(y)
     \\
    \ge&
    a^{-2}
    \min_{y\in \al X}
    \left\{
        -\varpi(y)p(y)^2 \sgn(\log \phi(t))
    \right\}
    |\log \phi(t) |
    -a^{-2}C\max_{y\in \al X}p(y).
    \end{aligned}
    \end{equation}

Thus, we can construct the lower bound
    \begin{equation*}
    \underline{\delta}:=
    \min
    \left\{
        \frac{1}{2}\min_{y\in \al X}\varrho_0(y),
        \exp\left(
        -\frac{2C \max_{y\in \al X}p(y)}
        {\min_{y\in \al X}\{\varpi(y)p(y)^2\}}
        \right)
    \right\}<1,
    \end{equation*}
    and assume by contradiction that
    \begin{equation*}
    \{t\in[0,T_{\max}) \mid \phi(t)<\underline{\delta} \,
    \}\neq \emptyset.
    \end{equation*}
Denote
    \begin{equation*}
    \tau:=\inf\{t\in[0,T_{\max}) \mid \phi(t)<\underline{\delta} \,
    \}<\infty.
    \end{equation*}
Then the continuity of $\phi$ gives 
    $\phi(\tau)=\underline{\delta}$,
    and the right derivative
    $\phi'_+ (\tau)\le0$.
Yet, the inequality \eqref{eq_heat_wellposed_2} implies
    \begin{equation*}
    \begin{aligned}
    \phi'_+ (\tau)
    \ge & \
    a^{-2}
    \min_{y\in \al X}
    \left\{
        -\varpi(y)p(y)^2 \sgn(\log \underline{\delta})
    \right\}
    |\log \underline{\delta} |
    -a^{-2}C\max_{y\in \al X}p(y)
    \\
    \ge & \
    a^{-2}
    \min_{y\in \al X}
    \left\{
        \varpi(y)p(y)^2
    \right\}
    \frac{2C \max_{y\in \al X}p(y)}
        {\min_{y\in \al X}\{\varpi(y)p(y)^2\}}
    -a^{-2}C\max_{y\in \al X}p(y)
    \\
    = & \ a^{-2}C\max_{y\in \al X}p(y) > 0,
    \end{aligned}
    \end{equation*}
    which is a contradiction to $\phi'_+ (\tau)\le0$.
Thus $\{t\in[0,T_{\max}) \mid \phi(t)<\underline{\delta} \,
    \}= \emptyset$.
That is
    \begin{equation*}
    \inf_{t\in[0,T_{\max})}\min_{x\in \al X}\varrho_t(x)\ge \underline{\delta},
    \end{equation*}
    i.e.,\ the solution has a positive lower bound.

\medskip
\noindent$\bullet$ \textbf{Conclusion.}
We have shown that the solution remains in the compact set
    \begin{equation*}
    \left\{
        u\in \al M \ \bigg |
        \min_{x\in \al X} u(x) \ge \underline{\delta},
        \ \text{and} \
        \mathbb{H}(u) \le \mathbb{H}(\varrho_0)
    \right \}.
    \end{equation*}
Hence no finite-time blow-up or loss of regularity can occur, 
    and the maximal time must be $T_{\max}=\infty$.
Uniqueness on $[0,\infty)$ follows from the local uniqueness.
\end{proof}

The next lemma, characterizing the eigenvalues of the Markov core $K$, will participate in the proof of Theorem \ref{thm_heatflow_converge}.
We collect the proof in Appendix \ref{sec_app_eigen_K}.
\begin{lemma}
\label{lem_eigen_K}
The eigenvalues of $K:\mathbb{R}^{\al X}\to \mathbb{R}^{\al X}$ are real.
Moreover, $1$ is the largest eigenvalue of $K$, with eigenvector $\mathbf{1}$.
\end{lemma}

Next, we present the following exponential convergence result for the heat flow.
\begin{theorem}
\label{thm_heatflow_converge}
The heat flow $\varrho_t$ converges exponentially fast to the equilibrium $\mathbf{1}$ in $\mathrm{L}^2(\al X,\varpi)$-distance.
\end{theorem}
\begin{proof}

We proceed with our proofs in three steps.

\medskip
\noindent$\bullet$ \textbf{Step 1. $\mathbf{1}$ is the unique equilibrium.}
The equilibrium condition of \eqref{eq_heat} is
    \begin{equation*}
    b^{-2}\Delta \varrho= a^{-2}\langle\log\varrho,p\rangle_\varpi\,p.
    \end{equation*}
Applying $\langle \mathbf{1},\cdot\rangle_\varpi$ to both sides and using
    $\langle \mathbf{1},\Delta\varrho \rangle_\varpi
    =-\langle \nabla \mathbf{1},\nabla \varrho\rangle_\varpi
    =0$
    gives $\langle\log\varrho,p\rangle_\varpi=0$.
Hence $\Delta\varrho=0$.
The integration–by–parts identity yields
    \begin{equation*}
    \langle \varrho,\Delta\varrho\rangle_\varpi
    =
    - \langle \nabla\varrho,\nabla\varrho\rangle_\varpi
    =
    -\frac{1}{2}
    \sum_{x,y\in \al X}
    (\varrho(x)-\varrho(y))^2K(x,y)\varpi(x)
    \le\ 0.
    \end{equation*}
When $\Delta\varrho=0$,
    $\varrho(y)=\varrho(x)$ for all edges with $K(x,y)>0$. 
Therefore $\varrho$ is constant on each connected component, 
    and by connectedness $\varrho$ is constant on the whole graph.
The condition 
    $\langle \log \varrho, p\rangle_\varpi=0$ then forces $\varrho\equiv 1$.
Thus $\mathbf{1}$ is the unique equilibrium of \eqref{eq_heat}.

\medskip
\noindent$\bullet$ \textbf{Step 2. A Łojasiewicz gradient inequality.}
By Lemma \ref{lem_heat_eq_wellposed},
    the trajectory $(\varrho_t)_{t\ge 0}$ is contained in a compact set,
    denoted as $\mathcal{K}$.
We claim that $\mathbb{H}$ satisfy the Łojasiewicz inequality on $\mathcal{K}$:
    \begin{equation*}
    \|\mathrm{grad}\,\mathbb{H}(\mu)\|_{\mathbf{g}_{\mu}}^2
    \ge
    C |
    \mathbb{H}(\mu)-\mathbb{H}(\mathbf{1})
    |,
    \end{equation*}
    where $C>0$ is a constant and 
    $\|\mathrm{grad}\,\mathbb{H}(\mu)\|_{\mathbf{g}_{\mu}}^2
    := 
    \mathbf{g}_\mu(\mathrm{grad}\,\mathbb{H}(\mu),\mathrm{grad}\,\mathbb{H}(\mu)).
    $

\medskip
\noindent$\star$ {\textbf{Step 2.1. Local expansion of the gradient norm.}}
{{Let $\mu=\textbf{1}+\boldsymbol{\varepsilon} $}} with small perturbation $\boldsymbol{\varepsilon}\in\mathbb{R}^{N}$. 
Using $\log(1+s)=s-\tfrac{1}{2}s^2+O(s^3)$
    and recalling that  the $\mathbf{g}$–gradient of $\bb H$ at $\mu$ is represented by
    \begin{equation*}
      D_{\mu}\,\mathrm{grad}\,\bb H(\mu)
      = \big(b^{-2}\nabla\log\mu,\;a^{-2}\langle \log\mu,p\rangle_{\varpi}\big),
    \end{equation*}
    we compute
    \begin{equation}
    \label{eq_heatflow_converge_1}
    \begin{aligned}
    \|\mathrm{grad}\,\mathbb{H}(\mu)\|_{\mathbf{g}_{\mu}}^2
    &=a^{-2}\langle\log\mu,p\rangle_\varpi^2
    +b^{-2}\langle\nabla\log\mu,\nabla\log\mu\rangle_\mu \\
    &=a^{-2}\langle\log\mu,p\rangle_\varpi^2
    +b^{-2}\langle\nabla\log\mu,\nabla\mu\rangle_\varpi \\
    &=a^{-2}\langle\log\mu,p\rangle_\varpi^2
    -b^{-2}\langle\log\mu,\Delta\mu\rangle_\varpi \\
    &=a^{-2}\langle\boldsymbol{\varepsilon},p\rangle_\varpi^2
    +b^{-2}\langle\boldsymbol{\varepsilon},(I-K)\boldsymbol{\varepsilon}\rangle_\varpi
    +O(\|\boldsymbol{\varepsilon}\|_{\varpi,2}^3).
    \end{aligned}
    \end{equation}

\noindent\textbf{\emph{Estimates of $\langle\boldsymbol{\varepsilon},(I-K)\boldsymbol{\varepsilon}\rangle_\varpi$.}}
In light of Lemma \ref{lem_eigen_K},
    let $1>\kappa_2\ge\kappa_3\ge\dots\ge\kappa_N$ be the eigenvalues of $K$ and $\mathbf{1},v_2,\dots,v_N$ be the corresponding eigenvectors that form an orthonormal eigenbasis:
    \begin{equation*}
    Kv_i=\kappa_i v_i,\quad 
    \langle v_i,v_j\rangle_\varpi=
    \left\{
        \begin{aligned}
        &1, && \text{if}\ i=j,
        \\
        &0, && \text{if}\ i\neq j.
        \end{aligned}
    \right.
    \end{equation*}
We can expand $\boldsymbol{\varepsilon}$ by the orthonormal eigenbasis of $K$:
    \begin{equation*}
    \boldsymbol{\varepsilon}=
    \langle
        \boldsymbol{\varepsilon},\mathbf{1}
    \rangle_\varpi\,
    \mathbf{1}
    +\sum_{i=2}^N c_i\,v_i,
    \qquad
    c_i:= \langle \boldsymbol{\varepsilon}, v_i\rangle.
    \end{equation*}
Then
    \begin{equation*}
    \begin{aligned}
    \langle\boldsymbol{\varepsilon},(I-K)\boldsymbol{\varepsilon}\rangle_\varpi
    =  
    \sum_{i=2}^N(1-\kappa_i)c_i^2
    \ge (1-\kappa_2)\sum_{i=2}^Nc_i^2
    =  
    (1-\kappa_2)
    \left\langle \sum_{i=2}^N c_iv_i, \sum_{i=2}^N c_iv_i \right\rangle.
    \end{aligned}
    \end{equation*}
Denote $\boldsymbol{\varepsilon}^\perp :=  \sum_{i=2}^N c_iv_i$
    and $\lambda_2:=1-\kappa_2>0$.
Then
    \begin{equation}
    \label{eq_heatflow_converge_2}
    \langle\boldsymbol{\varepsilon},(I-K)\boldsymbol{\varepsilon}\rangle_\varpi
    \ge
    \lambda_2 \|\boldsymbol{\varepsilon}^\perp\|_{\varpi,2}^2.
    \end{equation}

\noindent\textbf{\emph{Estimates of $\langle\boldsymbol{\varepsilon},p\rangle_\varpi^2$.}}
Let $p^\perp := p-\mathbf{1}$ and we have
    \begin{equation*}
    \begin{aligned}
    \langle\boldsymbol{\ep},p\rangle_\varpi^2
    = &
    \langle
    \boldsymbol{\ep}^\perp+\langle\boldsymbol{\ep},\mathbf{1}\rangle_\varpi \mathbf{1},
    p^\perp + \mathbf{1}
    \rangle_\varpi^2
    =
    (\langle
    \boldsymbol{\ep}^\perp,p^\perp
    \rangle_\varpi
    +
    \langle
    \boldsymbol{\ep},\mathbf{1}
    \rangle_\varpi)^2
    \\
    = &
    \langle
    \boldsymbol{\ep}^\perp,p^\perp
    \rangle_\varpi^2
    +
    \langle
    \boldsymbol{\ep},\mathbf{1}
    \rangle_\varpi^2 
    +
    2
    \langle
    \boldsymbol{\ep}^\perp,p^\perp
    \rangle_\varpi
    \langle
    \boldsymbol{\ep},\mathbf{1}
    \rangle_\varpi.
    \end{aligned}
    \end{equation*}
For any $\eta,\alpha,\beta>0$,
    we have the elementary inequality
    $2\alpha\beta \ge -\eta \alpha^2-\eta^{-1}\beta^2$, 
    so
    \begin{equation*}
    \begin{aligned}
    \langle\boldsymbol{\ep},p\rangle_\varpi^2
    \ge
    (1-\eta)
    \langle
    \boldsymbol{\ep},\mathbf{1}
    \rangle_\varpi^2 
    +
    (1-\eta^{-1})
    \langle
    \boldsymbol{\ep}^\perp,p^\perp
    \rangle_\varpi^2.
    \end{aligned}
    \end{equation*}
Hence for $\eta\in(0,1)$, $1-\eta^{-1}<0$ and by Cauchy-Schwarz inequality,
    \begin{equation}
    \label{eq_heatflow_converge_3}
    \langle\boldsymbol{\ep},p\rangle_\varpi^2
    \ge
     (1-\eta)
    \langle
    \boldsymbol{\ep},\mathbf{1}
    \rangle_\varpi^2 
    +
    (1-\eta^{-1})
    \| \boldsymbol{\ep}^\perp\|_{\varpi,2}^2 \| p^\perp\|^2_{\varpi,2}.
    \end{equation}

Substituting \eqref{eq_heatflow_converge_2} and \eqref{eq_heatflow_converge_3} into \eqref{eq_heatflow_converge_1}, 
    we get for $\eta\in(0,1)$,
    \begin{equation*}
    \begin{aligned}
    \|\mathrm{grad}\,\mathbb{H}(\mu)\|_{\mathbf{g}_{\mu}}^2
    \ge& \
    a^{-2}\!\left((1-\eta)\langle\boldsymbol{\ep},\mathbf{1}\rangle_\varpi^2
    +(1-\eta^{-1})\|\boldsymbol{\ep}^\perp\|_{\varpi,2}^2\|p^\perp\|_{\varpi,2}^2\right)
    +b^{-2}\lambda_2\|\boldsymbol{\ep}^\perp\|_{\varpi,2}^2
    +O(\|\boldsymbol{\ep}\|_{\varpi,2}^3)\\
    =&\ 
    a^{-2}(1-\eta)\langle\boldsymbol{\ep},\mathbf{1}\rangle_\varpi^2
    +\big(b^{-2}\lambda_2+a^{-2}(1-\eta^{-1})\|p^\perp\|_{\varpi,2}^2\big)
    \|\boldsymbol{\ep}^\perp\|_{\varpi,2}^2
    +O(\|\boldsymbol{\ep}\|_{\varpi,2}^3).
    \end{aligned}
    \end{equation*}
In the case $p\neq\mathbf{1}$, that is, when $\|p^\perp\|_{\varpi,2}^{-2}\neq 0$, we choose
    \begin{equation*}
    \eta=\frac{2+a^2b^{-2}\lambda_2\|p^\perp\|_{\varpi,2}^{-2}}
    {2+2a^2b^{-2}\lambda_2\|p^\perp\|_{\varpi,2}^{-2}}\in(0,1).
    \end{equation*}
Then, we have for a constant $C_1>0$,
    \begin{equation}
    \begin{aligned}
    \label{eq_taylor_gradH_at1}
     \|\mathrm{grad}\,\mathbb{H}(\mu)\|_{\mathbf{g}_{\mu}}^2
    \ge &
    \frac{b^{-2}\lambda_2\| p^\perp \|_{\varpi,2}^{-2}}
    {2+2a^2 b^{-2}\lambda_2\| p^\perp \|_{\varpi,2}^{-2}}
    \langle
    \boldsymbol{\ep},\mathbf{1}
    \rangle_\varpi^2 
    +
    \frac{
    \left(1+a^{2}b^{-2}\lambda_2\| p^\perp \|_{\varpi,2}^{-2}
    \right)
    b^{-2}\lambda_2
    }
    {2+a^2 b^{-2}\lambda_2\| p^\perp \|_{\varpi,2}^{-2}}
    \| \boldsymbol{\ep}^\perp\|_{\varpi,2}^2
    +
    O(\| \boldsymbol{\ep} \|_{\varpi,2}^3)
    \\
    \ge & \
    C_1(
    \langle
    \boldsymbol{\ep},\mathbf{1}
    \rangle_\varpi^2  
    +
     \| \boldsymbol{\ep}^\perp\|_{\varpi,2}^2
    )
    +
    O(\| \boldsymbol{\ep} \|_{\varpi,2}^3)
    =
    C_1
    \| \boldsymbol{\ep}\|_{\varpi,2}^2
    +
    O(\| \boldsymbol{\ep} \|_{\varpi,2}^3).
    \end{aligned}
    \end{equation}
In the degenerate case $p=\mathbf{1}$, i.e., when $\|p^\perp\|_{\varpi,2}^{-2}=0$, 
    we may simply take $\eta=\tfrac{1}{2}$ and choose
    \begin{equation*}
    C_1 := \min\bigl\{a^{-2}/2,\;b^{-2}\lambda_2\bigr\}.
    \end{equation*}

\medskip
\noindent$\star$ {\textbf{Step 2.2. Local expansion of the entropy.}}
Taylor expansion gives
    \begin{equation}
    \label{eq_taylor_bbH_at1}
    \begin{aligned}
    \mathbb{H}(\mu)-\mathbb{H}(\mathbf{1})
    = & \
    \sum_x
    (
    \mu(x)\log\mu(x)-\mu(x)
    )\varpi(x)
    -
    \sum_x
    (-1)\varpi(x)
    \\
    = & \
    \sum_x
    (
    \mu(x)\log\mu(x)-\mu(x)+1
    )\varpi(x)
    \\
    = & \
    \sum_x
    \big(
    (1+\boldsymbol{\ep}(x))\log(1+\boldsymbol{\ep}(x))-\boldsymbol{\ep}(x)
    \big)\varpi(x)
    \\
    = & \
    \sum_x
    \bigg(
    (1+\boldsymbol{\ep}(x))
    \bigg(\boldsymbol{\ep}(x)-\frac{1}{2}\boldsymbol{\ep}(x)^2+O(\boldsymbol{\ep}(x)^3) \bigg)
    -\boldsymbol{\ep}(x)
    \bigg)\varpi(x)
    \\
    = & \
    \sum_x
    \bigg(
    \frac{1}{2} \boldsymbol{\ep}(x)^2+ O(\boldsymbol{\ep}(x)^3)
    \bigg)
    \varpi(x)
    \\
    = & \
    \frac{1}{2} 
    \| \boldsymbol{\ep} \|_{\varpi,2}^2
    +
    O(  \| \boldsymbol{\ep} \|_{\varpi,2}^3 ).
    \end{aligned}
    \end{equation}
    
Comparing \eqref{eq_taylor_gradH_at1} and \eqref{eq_taylor_bbH_at1}, there exists a neighborhood $\mathcal U$ of $\mathbf{1}$ where
    \begin{equation*}
    \|\mathrm{grad}\,\mathbb{H}(\mu)\|_{\mathbf{g}_{\mu}}^2
    \ge C_1\,|\mathbb{H}(\mu)-\mathbb{H}(\mathbf{1})|.
    \end{equation*}
That is, $\mathbb{H}$ satisfies a local Łojasiewicz gradient inequality around $\mathbf{1}$ with exponent $\tfrac12$.

\medskip
\noindent$\star$ {\textbf{Step 2.3. A global {\L}ojasiewicz gradient inequality.}}
Recall that $\mathcal{K}$ is the compact invariant set of $(\varrho_t)_{t\ge 0}$.
Since the continuous functions
    $\mu\mapsto\|\mathrm{grad}\,\mathbb{H}(\mu)\|_{\mathbf{g}_\mu}^2$ and
    $\mu\mapsto|\mathbb{H}(\mu)-\mathbb{H}(\mathbf{1})|$
    vanish only at $\mathbf{1}$, 
    the continuous ratio
    \begin{equation*}
    \mu \;\longmapsto\; 
    \frac{\|\mathrm{grad}\,\mathbb{H}(\mu)\|_{\mathbf{g}_\mu}^2}
       {|\mathbb{H}(\mu)-\mathbb{H}(\mathbf{1})|}
    \end{equation*}
    attains a positive minimum $C_2>0$ on the compact set $\mathcal K\setminus\mathcal U$.
Setting $C:=\min\{C_1,C_2\}$ we obtain, for every $\mu\in\mathcal K$,
\begin{equation*}
  \|\mathrm{grad}\,\mathbb{H}(\mu)\|_{\mathbf{g}_\mu}^2
  \;\ge\;
  C\,\big|\mathbb{H}(\mu)-\mathbb{H}(\mathbf{1})\big|.
\end{equation*}

\medskip
\noindent$\bullet$ \textbf{Step 3. Exponential convergence.}
The Łojasiewicz inequality yields
    \begin{equation*}
    \frac{\di}{\di t}\big(\mathbb{H}(\varrho_t)-\mathbb{H}(\mathbf{1})\big)
    =-\|\mathrm{grad}\,\mathbb{H}(\varrho_t)\|_{\mathbf{g}_{\varrho_t}}^2
    \le -C\big(\mathbb{H}(\varrho_t)-\mathbb{H}(\mathbf{1})\big),
    \end{equation*}
    whose solution gives
    \begin{equation*}
    \mathbb{H}(\varrho_t)-\mathbb{H}(\mathbf{1})
    \le e^{-Ct}\big(\mathbb{H}(\varrho_0)-\mathbb{H}(\mathbf{1})\big).
    \end{equation*}

Define $\iota:[0,\infty)\to\mathbb R$ by
    \begin{equation*}
    \iota(u)=u\log u-u+1 .
    \end{equation*}
Since $\iota(1)=0$ and $\iota'(1)=0$, Taylor's formula with integral remainder at $1$ gives
    \begin{equation*}
    \iota(u)=\iota(1)+\iota'(1)(u-1)+\int_1^{u} (u-r)\,\iota''(r)\,\mathrm dr
    =\int_1^{u} (u-r)\,\iota''(r)\,\mathrm dr,
    \qquad\text{with}\qquad
\iota''(r)=\frac{1}{r}.
    \end{equation*}
Since $(\varrho_t)_{t\ge 0}$ is contained in a compact set $\mathcal K\subset\mathcal M_+$, there exist constants $\overline{\delta}>0$ such that for all $t\ge0$ and all $x\in\mathcal X$
    \begin{equation*}
    0<\varrho_t(x)<\overline{\delta} .
    \end{equation*}
Fix $u\in(0,\overline{\delta})$.
For every $r$ on the segment joining $1$ and $u$ we have                 $\iota''(r)=1/r\ge 1/\overline{\delta}$.
Hence
    \begin{equation*}
    \iota(u)\ge \frac{1}{\overline{\delta}}\int_1^{u} (u-r)\,\mathrm dr
    =\frac{1}{2\overline{\delta}}(u-1)^2.
    \end{equation*}
Applying this with $u=\varrho_t(x)$ yields
    \begin{equation*}
    \frac{1}{2\overline{\delta}}\big(\varrho_t(x)-1\big)^2
    \le \iota\big(\varrho_t(x)\big)
    .
    \end{equation*}
Multiplying by $\varpi(x)$ and summing over $x$ yields
    \begin{equation*}
    \frac{1}{2\overline{\delta}}\,\|\varrho_t-\mathbf 1\|_{2,\varpi}^2
    \le \mathbb H(\varrho_t)-\mathbb H(\mathbf 1) .
    \end{equation*}
Combining this with the entropy decay estimate gives
    \begin{equation*}
    \|\varrho_t-\mathbf 1\|_{2,\varpi}^2
    \le 2\overline{\delta}\,e^{-Ct}\big(\mathbb H(\varrho_0)-\mathbb H(\mathbf 1)\big) .
    \end{equation*}
\end{proof}

\section{Conclusions}
\label{sec_conc}
We proposed and analyzed a Wasserstein-like metric \(\al W^{a,b}_p\) 
    for nonnegative measures on finite reversible Markov chains, 
    incorporating a source constrained to a fixed direction \(p\). 
On \(\al M_+\), the induced Riemannian structure
    aligns the length distance with \(\al W^{a,b}_p\),  
    and identifies the generalized heat equation as the entropy gradient flow, 
    with global well-posedness and exponential convergence to equilibrium. 
However, there are still many directions that need to be further explored, 
    such as interpreting equations like Fokker--Planck equations with source as gradient of an energy functional. 
We leave these interesting questions for future work.

\section*{Acknowledgments}
The work of X. Xue is supported by the Natural Science Foundation of China (grants 12271125), and the work of X. Wang was supported by the Natural Science Foundation of China (grants 123B2003), the China Postdoctoral Science Foundation (grants 2025M774290), and Heilongjiang Province Postdoctoral Funding (grants LBH-Z24167).

\section*{Conflict of interest statement}
The authors declare no conflicts of interest.

\section*{Data availability statement}
The data supporting the findings of this study are available from the corresponding author upon reasonable request.

\section*{Ethical statement}
The authors declare that this manuscript is original, has not been published before, and is not currently being considered for publication elsewhere. The study was conducted by the principles of academic integrity and ethical research practices. All sources and contributions from others have been properly acknowledged and cited. The authors confirm that there is no fabrication, falsification, plagiarism, or inappropriate manipulation of data in the manuscript.

\appendix
\addtocontents{toc}{\protect\setcounter{tocdepth}{1}} 
\section{Technical proofs}

\subsection{Proof of Lemma \ref{lem_ELinSq_potential_rewrite}}
\label{sec_app_a}
We argue like Theorem 5.4 in \cite{dolbeault2009new}. 

Take $(\mu_t,\psi_t,h_t)_{t\in[0,1]}\in \text{CE}_p(\mu_0,\mu_1;[0,1])$ and $\ep>0$. 
Define $\mathbf s_\ep:[0,1]\rightarrow [0, C_\ep]$ as
    \begin{equation*}
    \mathbf s_\ep(t):= \int_0^t (\ep + a^2h_\tau^2+b^2[A_{\mu_\tau}\psi_\tau,\psi_\tau])^{\frac{1}{2}} \di \tau,
    \end{equation*}
    where $C_\ep:= \int_0^1 (\ep+a^2h_\tau^2+b^2[A_{\mu_\tau}\psi_\tau,\psi_\tau])^{\frac{1}{2}} \di \tau$.
$\mathbf s_\ep$ is invertible, the inverse of which is denoted as $\mathbf{t}_\ep:[0,C_\ep]\to[0,1]$ and we have
    \begin{equation*}
    \dot{\mathbf{t}}_\ep(\mathbf s_\ep(t))=(\ep+a^2h_t^2+b^2[A_{\mu_t}\psi_t,\psi_t])^{-\frac{1}{2}}.
    \end{equation*}
Let $\mu^\ep:=\mu\circ \mathbf{t}_\ep,
    \psi^\ep:=\dot{\mathbf{t}}_\ep \times (\psi\circ \mathbf {t}_\ep),
    h^\ep:=\dot{\mathbf{t}}_\ep \times ( h\circ \mathbf {t}_\ep)$, 
    then $(\mu^\ep_s,\psi^\ep_s,h^\ep_s)_{s\in[0,C_\ep]}\in \text{CE}_p(\mu_0,\mu_1;[0, C_\ep])$
    as the continuity equation and the end-point conditions are satisfied
    \begin{equation*}
    \begin{gathered}
    \dot \mu^\ep_s=  \dot{\mathbf{t}}_\ep(s)(\dot \mu\circ{\mathbf{t}_\ep(s)})
    = \dot{\mathbf{t}}_\ep(s) (B_{\mu\circ{\mathbf{t}_\ep(s)}}(\psi\circ \mathbf{t}_\ep(s)) + (h\circ \mathbf{t}_\ep(s)) p)
    =  B_{\mu^\ep_s}\psi^\ep_s + h^\ep_sp ,\\
    \mu^\ep_0= \mu\circ \mathbf{t}_\ep(0) = \mu_0,
    \quad
    \mu^\ep_{C_\ep}= \mu\circ \mathbf{t}_\ep(C_\ep) = \mu_1.
    \end{gathered}
    \end{equation*}
And let $\mathbf{T}(\tau):=C_\ep \tau$ for $\tau\in [0,1]$, 
    then 
    \begin{equation*}
    \begin{gathered}
    \frac{\di}{\di\tau} \mu^\ep\circ \mathbf{T}(\tau)
    =
    C_\ep \dot\mu^\ep\circ \mathbf{T}(\tau)
    =
    C_\ep (B_{\mu^\ep\circ \mathbf{T}(\tau)} \psi^\ep\circ \mathbf{T}(\tau) +h^\ep\circ \mathbf{T}(\tau)p),
    \\
    \mu^\ep\circ \mathbf T(0)=\mu^\ep_0=\mu_0,
    \quad
    \mu^\ep\circ \mathbf T(1)=\mu^\ep_{C_\ep}=\mu_1.
    \end{gathered}
    \end{equation*}
As a result, $(\mu^\ep\circ \mathbf{T}(\tau),C_\ep   \psi^\ep \circ \mathbf{T}(\tau),C_\ep   h^\ep \circ \mathbf{T}(\tau))_{\tau\in[0,1]}\in \text{CE}_p(\mu_0,\mu_1;[0,1])$.
Thus, by the definition of $\al W_p^{a,b}$ we have
    \begin{equation*}
    \begin{aligned}
    \al W_p^{a,b}(\mu_0,\mu_1)
    \le &
    \left(
        E^{a,b}_{\rm Quad}((\mu^\ep\circ \mathbf{T}(\tau),C_\ep   \psi^\ep \circ \mathbf{T}(\tau),C_\ep   h^\ep \circ \mathbf{T}(\tau))_{\tau\in[0,1]} )
    \right)^{\frac{1}{2}}
    \\
    = &
    a^2\int_0^1 C_\ep^2 (h^\ep \circ \mathbf{T}(\tau))^2 \di \tau
    + b^2 \int_0^1 C_\ep^2 [A_{\mu^\ep \circ \mathbf{T}(\tau)}(\psi^\ep \circ \mathbf{T}(\tau)),\psi^\ep \circ \mathbf{T}(\tau)] \di \tau
    \\=& C_\ep 
    \int_0^{C_\ep}  a^2 (h^\ep (s))^2 + b^2 [A_{\mu^\ep(s)}\psi^\ep (s),\psi^\ep (s)] \di s
    \\=&
    C_\ep
    \int_0^1 \dot{\mathbf{t}}_\ep(\mathbf s_\ep(t)) \bigg\{a^2  ( h (t))^2 
    + b^2 [A_{\mu(t)}\psi (t),\psi (t)]\bigg\} \di t
    \\=&
    \int_0^1 (\ep+a^2h^2_t+b^2[A_{\mu_t}\psi_t,\psi_t])^{\frac{1}{2}} \di t \\
    & \times
    \int_0^1 \left(
    \frac{a^2h^2_t+b^2[A_{\mu_t}\psi_t,\psi_t]}{\ep+a^2h^2_t+b^2[A_{\mu_t}\psi_t,\psi_t]}
    \right)^{\frac{1}{2}} \bigg\{a^2h^2_t+b^2[A_{\mu_t}\psi_t,\psi_t]\bigg\}^{\frac{1}{2}} \di t.
    \end{aligned}
    \end{equation*}
Letting $\ep \rightarrow 0$, we see that
    \begin{equation*}
    \al W_p^{a,b}(\mu_0,\mu_1)\le
    \int_0^1\bigg\{
    a^2h^2_t+b^2[A_{\mu_t}\psi_t,\psi_t]
    \bigg\}^{\frac{1}{2}}\di t
    =
    \left( E_{\rm LinSq}^{a,b}((\mu_t,\psi_t,h_t)_{t\in[0,1]}) \right)^{\frac{1}{2}}
    .
    \end{equation*}
On the other hand, Jensen's inequality gives
    \begin{equation*}
    E_{\rm Quad}^{a,b}((\mu_t,\psi_t,h_t)_{t\in[0,1]}) \ge E_{\rm LinSq}^{a,b}((\mu_t,\psi_t,h_t)_{t\in[0,1]}).
    \end{equation*}
Therefore, we obtain the desired rewrite for $ \al W_p^{a,b}$
    \begin{equation*}
    \al W_p^{a,b}(\mu_0,\mu_1)^2
    = \inf 
    \left\{
    E_{\rm LinSq}^{a,b}((\mu_t,\psi_t,h_t)_{t\in[0,1]})
    ~\bigg|~
    (\mu_t,\psi_t,h_t)_{t\in[0,1]}\in \text{CE}_p(\mu_0,\mu_1;[0,1])
    \right\}.
    \end{equation*}

\subsection{Proof of Lemma \ref{lem_l1_bound}}
\label{sec_app_c}
We verify the upper and lower bounds separately in two steps.

\noindent$\bullet$  \textbf{Step 1. $\;{\al W}_p^{a,b}(\mu_0,\mu_1) \le \mathsf C\,\|\mu_1-\mu_0\|_{\varpi,1}$.}
For the triplet $(\mu_t^\star,\psi_t^\star,h_t^\star)_{t\in[0,1]} $ constructed in \eqref{exm_ce_non_empty}–\eqref{exm_ce_non_empty'}, 
    its LinSq–action satisfies
    \begin{equation*}
    \begin{aligned}
    \Big( E_{\rm LinSq}^{a,b}\big((\mu_t^\star,\psi_t^\star,h_t^\star)_{t\in[0,1]}\big) \Big)^{\!1/2}
    &= a\int_0^{1/3} |h_t^\star|\di t
      + b\int_{1/3}^{2/3} [A_{\mu_t^\star}\psi_t^\star,\psi_t^\star]^{1/2} \di t
      + a\int_{2/3}^{1} |h_t^\star|\di t
    \\
    &\le b \int_{1/3}^{2/3}
       \|A_{\mu_t^\star}\|^{1/2}\,
       \big\| (B_{\mu_t^\star}\!\restriction_{\mathsf{H}})^{-1}\big\|\,
       \big\|(\mu_1-\mu_0)+[\mu_0-\mu_1,\varpi]\,p\big\|_2 \, \di t
    \\
    &\quad +\, a\big([\mu_1-\mu_0,\varpi]+2\varepsilon\big),
    \end{aligned}
    \end{equation*}
    where we used that $h_t^\star$ and $\psi_t^\star$ acts separately.
Since along the path $\|\mu_t^\star\|_{\varpi,1}\le \|\mu_1\|_{\varpi,1}+\varepsilon$, we obtain
    \begin{equation*}
    \begin{aligned}
    \Big( E_{\rm LinSq}^{a,b}\big((\mu_t^\star,\psi_t^\star,h_t^\star)_{t\in[0,1]}\big) \Big)^{\!1/2}
    &\le b\,\max_{\substack{\nu\in\al M\\ \|\nu\|_{\varpi,1}\le \|\mu_1\|_{\varpi,1}+\varepsilon}}
         \Big\{\|A_{\nu}\|^{1/2}\,\big\|(B_{\nu}\!\restriction_{\mathsf{H}})^{-1}\big\|\Big\}
         \,\big\|\,(\mu_1-\mu_0)+[\mu_0-\mu_1,\varpi]\,p\,\big\|_2
    \\
    &\quad +\, a\big([\mu_1-\mu_0,\varpi]+2\varepsilon\big).
    \end{aligned}
    \end{equation*}
Denote $\min\!\varpi:=\min_{x\in\al X}\varpi(x)$. 
Then
    \begin{equation*}
    \big\|\,(\mu_1-\mu_0)+[\mu_0-\mu_1,\varpi]\,p\,\big\|_2
    \;\le\;
    \|\mu_1-\mu_0\|_2 + \|\mu_1-\mu_0\|_{\varpi,1}\,\|p\|_2
    \;\le\;
    \Big({\min\!\varpi}^{-1}+\|p\|_2\Big)\,\|\mu_1-\mu_0\|_{\varpi,1},
    \end{equation*}
    where we used $\|\mu_1-\mu_0\|_2\le \min\!\varpi^{-1}\,\|\mu_1-\mu_0\|_{\varpi,1}$.
Combining the estimates and letting $\varepsilon\downarrow 0$ yields
    \begin{equation*}
    {\al W}_p^{a,b}(\mu_0,\mu_1)\;\le\; \mathsf C\,\|\mu_1-\mu_0\|_{\varpi,1},
    \end{equation*}
    with
    \begin{equation*}
    \mathsf C
    := \max\!\left\{
    a
    + b\Big(\min\!\varpi^{-1}+\|p\|_2\Big)\,
    \|A_{\nu}\|^{1/2}\,\big\|(B_{\nu}\!\restriction_{\mathsf{H}})^{-1}\big\|
    \;\middle|\;
    \nu\in \al M,\ \|\nu\|_{\varpi,1}\le \max_{i=0,1} \|\mu_i\|_{\varpi,1}
    \right\}.
    \end{equation*}
    
\vspace{0.2cm}

\noindent$\bullet$ 
\textbf{Step 2. $\; \|\mu_1-\mu_0\|_{\varpi,1} \le \mathsf c\,{\al W}_p^{a,b}(\mu_0,\mu_1)$.}
With a minor adaptation to account for possibly unbounded total mass in $\al M$, 
    the argument of Lemma~3.10 in \cite{maas2011gradient} carries over.
For $n\in\mathbb{N}_+$, let $(\mu^n_t,\psi^n_t,h^n_t)_{t\in[0,1]}\in\mathrm{CE}_p(\mu_0,\mu_1;[0,1])$ 
    be such that
    \begin{equation*}
    \Big(E_{\rm LinSq}^{a,b}\big((\mu^n_t,\psi^n_t,h^n_t)_{t\in[0,1]}\big)\Big)^{\!1/2}
    = \int_0^1\!\Big(a^2|h_t^n|^2 + b^2\,[A_{\mu_t^n}\psi_t^n,\psi_t^n]\Big)^{1/2}\di t
    \;\le\; {\al W}_p^{a,b}(\mu_0,\mu_1) + n^{-1}.
    \end{equation*}
Using $[\varpi,B_{\mu_t^n}\psi_t^n]=0$, we estimate the total mass along the path:
    \begin{equation*}
    \begin{aligned}
    \max_{t\in[0,1]}[\mu^n_t,\varpi]
    &\le [\mu_0,\varpi] + \int_0^1\Big|\tfrac{\di}{\di t}[\mu_t^n,\varpi]\Big|\di t
    = [\mu_0,\varpi] + \int_0^1 \big|[\,\varpi,\,\dot\mu_t^n]\big|\di t
    \\
    &= [\mu_0,\varpi] + \int_0^1 \big|[\,\varpi,\,h_t^n\,p]\big|\di t
    \;\le\; [\mu_0,\varpi] + \int_0^1 |h_t^n|\di t
    \\
    &\le [\mu_0,\varpi] + a^{-1}\Big(E_{\rm LinSq}^{a,b}\big((\mu^n_t,\psi^n_t,h^n_t)\big)\Big)^{\!1/2}
    \\
    &\le [\mu_0,\varpi] + a^{-1}{\al W}_p^{a,b}(\mu_0,\mu_1) + a^{-1}n^{-1}.
    \end{aligned}
    \end{equation*}
    Hence
    \begin{equation*}
    \max_{y\in\al X}\ \max_{t\in[0,1]}\mu_t^n(y)
    \;\le\;
    c_0:=\frac{[\mu_0,\varpi] + a^{-1}{\al W}_p^{a,b}(\mu_0,\mu_1) + a^{-1}n^{-1}}{\min\!\varpi}.
    \end{equation*}
Let $\varphi\in\mathbb{R}^{\al X}$ be arbitrary. 
Then using the continuity equation for $(\mu^n_t,\psi^n_t,h^n_t)_{t\in[0,1]}$ and $\Pi B_{\mu^n_t} = A_{\mu^n_t}$,
    \begin{equation*}
    \begin{aligned}
    \big|[\Pi\varphi,\mu_1-\mu_0]\big|
    &=\left|\int_0^1 [\Pi\varphi,\dot\mu_t^n]\di t\right|
    =\left|\int_0^1 [\Pi\varphi,B_{\mu_t^n}\psi_t^n + h_t^n p]\di t\right|
    \\
    &\le \left|\int_0^1 [\Pi\varphi,B_{\mu_t^n}\psi_t^n]\di t\right|
         + \left|\int_0^1 [\Pi\varphi,h_t^n p]\di t\right|
    \\
    &=\left|\int_0^1 [\varphi,A_{\mu_t^n}\psi_t^n]\di t\right|
      + \|\varphi\|_\infty\int_0^1 |h_t^n|\di t
    \\
    &\le \int_0^1 [\varphi,A_{\mu_t^n}\varphi]^{1/2}\,[\psi_t^n,A_{\mu_t^n}\psi_t^n]^{1/2}\di t
       + \|\varphi\|_\infty\int_0^1 |h_t^n|\di t.
    \end{aligned}
    \end{equation*}
Moreover,
    \begin{equation*}
    [\varphi,A_{\mu_t^n}\varphi]
    =\frac12\sum_{x,y}(\varphi(y)-\varphi(x))^2 K(x,y)\,\hat\mu_t^n(x,y)\,\varpi(x)
    \;\le\; c_1\,\|\varphi\|_\infty^2,
    \end{equation*}
    with $c_1:=2\,\max_{r,s\in[0,c_0]}\theta(r,s)=2c_0$, 
    using $(\varphi(y)-\varphi(x))^2\le 4\|\varphi\|_\infty^2$ and $\sum_y K(x,y)=1$.
Therefore,
    \begin{equation*}
    \begin{aligned}
    \big|[\Pi\varphi,\mu_1-\mu_0]\big|
    &\le \|\varphi\|_\infty\left( c_1^{1/2}\int_0^1 [\psi_t^n,A_{\mu_t^n}\psi_t^n]^{1/2} \di t
    + \int_0^1 |h_t^n|\di t \right)
    \\
    &\le \|\varphi\|_\infty\,(a^{-1}+c_1^{1/2}b^{-1})
          \Big(E_{\rm LinSq}^{a,b}\big((\mu_t^n,\psi_t^n,h_t^n)\big)\Big)^{\!1/2}
    \\
    &\le \|\varphi\|_\infty\,(a^{-1}+c_1^{1/2}b^{-1})
          \big({\al W}_p^{a,b}(\mu_0,\mu_1)+n^{-1}\big).
    \end{aligned}
    \end{equation*}
Taking the supremum over $\|\varphi\|_\infty\le 1$ yields
    \begin{equation}
    \label{ineq_l1_low_bound_n}
    \|\mu_1-\mu_0\|_{\varpi,1}
    \;\le\; (a^{-1}+c_1^{1/2}b^{-1})\big({\al W}_p^{a,b}(\mu_0,\mu_1)+n^{-1}\big).
    \end{equation}
Letting $n\to\infty$ gives
    \begin{equation*}
    \|\mu_1-\mu_0\|_{\varpi,1}
    \;\le\; \mathsf c\,{\al W}_p^{a,b}(\mu_0,\mu_1),
    \qquad
    \mathsf c:= a^{-1}+b^{-1}\sqrt{\frac{2\,[\mu_0,\varpi]+2\,a^{-1}{\al W}_p^{a,b}(\mu_0,\mu_1)}{\min\!\varpi}}.
    \end{equation*}

\subsection{Proof of Lemma \ref{lem_positive_admissible_set}}
\label{sec_app_g}

Our argument adapts Lemma~3.30 in \cite{maas2011gradient}.
Fix $\varepsilon\in(0,1)$. 
By Lemma \ref{lem_ELinSq_potential_rewrite} we can choose 
    $(\mu_t,\psi_t,h_t)_{t\in[0,1]}\in \mathrm{CE}_{p}(\mu_0,\mu_1;[0,1])$ such that
    \begin{equation*}
    ({E}^{a,b}_{\rm LinSq}((\mu_t,\psi_t,h_t)_{t\in[0,1]}))^{\frac{1}{2}}
    =
    \int_{0}^{1}
    \sqrt{a^{2} h_t^{2}+b^{2}\langle \nabla\psi_t,\nabla \psi_t\rangle_{\mu_t} } \di t
    \;\le\;  \mathcal{W}^{a,b}_{p}(\mu_0,\mu_1)+\varepsilon.
    \end{equation*}
Introduce the strictly positive endpoints 
    $\mu_i^{\varepsilon}:=(1-\varepsilon)\mu_i+\varepsilon\,\mathbf{1}\in\al M_+$, $i\in\{0,1\}$, 
    and the perturbed terms
    \begin{equation*}
    \mu_t^{\varepsilon}(x):=(1-\varepsilon)\mu_t(x)+\varepsilon,
    \qquad
    \Psi_t^{\varepsilon}(x,y):=(1-\varepsilon)
    \frac{\nabla\psi_t(x,y)\hat\mu_t(x,y)}{\hat\mu_t^\ep(x,y)},
    \qquad
    h_t^{\varepsilon}:=(1-\varepsilon)h_t.
    \end{equation*}
Here $\Psi_t^\ep$ is well-defined because $\mu_t^\ep$ is strictly positive.

Recall that $\mathscr{G}_{\mu_t^\ep}$ consists of 
    equivalence classes of functions in $\mathbb{R}^{\al X\times \al X}$ 
    that agree at every pair $(x,y)$ with ${\hat\mu_t^\ep}(x,y)K(x,y)>0$.
Equipped with $\langle\cdot,\cdot\rangle_{\mu_t^\ep}$, 
    the space $\mathscr{G}_{\mu_t^\ep}$ is a Hilbert space.
The gradient operator $\nabla:\mathrm L^2(\al X,\varpi)\to \mathscr{G}_{\mu_t^\ep}$ is 
    a well defined linear map induced by the quotient, 
    and its negative adjoint is 
    $\nabla_{\mu_t^\ep}:\mathscr{G}_{\mu_t^\ep}\to \mathrm L^2(\al X,\varpi)$.
We have the orthogonal decomposition
    \begin{equation}
    \label{eq_orthogonal_decomposition_gmu}
    \mathscr{G}_{\mu_t^\ep}=\mathrm{Ran}(\nabla)\oplus^{\perp}\mathrm{Ker}(\nabla_{\mu_t^\ep}\cdot).
    \end{equation}
Let $\mathscr P_{\mu_t^\ep}$ be the orthogonal projection in $\mathscr{G}_{\mu_t^\ep}$ onto $\mathrm{Ran}(\nabla)$.
There exists a measurable function $\psi^\ep_t:[0,1]\to \mathbb{R}^{\al X}$ such that $\mathscr P_{\mu_t^\ep}\Psi_t^\ep=\nabla \psi_t^\ep$.
By \eqref{eq_orthogonal_decomposition_gmu},
    \begin{equation*}
    \nabla_{\mu_t^\ep}\cdot \nabla\psi_t^\ep
    =
    \nabla_{\mu_t^\ep}\cdot\big(\Psi_t^\ep-\mathscr P_{\mu_t^\ep}\Psi_t^\ep\big)
    +\nabla_{\mu_t^\ep}\cdot\big(\mathscr P_{\mu_t^\ep}\Psi_t^\ep\big)
    =\nabla_{\mu_t^\ep} \cdot \Psi_t^\ep.
    \end{equation*}
Thus from the definition of $\Psi_t^\ep$, 
    for any $x\in\al X$,
    \begin{align*}
    (\nabla_{\mu_t^\ep}\cdot \nabla\psi_t^\ep)(x)
    = & 
    \frac{1}{2}\sum_{y\in \al X}
    \big(\Psi_t^\ep(x,y)-\Psi_t^\ep(y,x)\big)
    \hat \mu_t^\ep(x,y)K(x,y)
    \\
    = & \frac{1-\ep}{2}\sum_{y\in \al X}
    \big(\nabla\psi_t(x,y)-\nabla\psi_t(y,x)\big)
    \hat \mu_t(x,y)K(x,y)
    =
    (1-\ep) (\nabla_{\mu_t}\cdot \nabla\psi_t)(x).
    \end{align*}
Hence the potential continuity equation \eqref{eq_ce_discrete_calculus} holds:
    \begin{equation*}
    \dot \mu^\ep_t+\nabla_{\mu_t^\ep}\cdot \nabla \psi_t^\ep
    =(1-\ep) 
    ( \dot \mu_t+\nabla_{\mu_t}\cdot \nabla \psi_t )
    =
    (1-\ep)h_t p
    =
    h_t^\ep\,p ,
    \end{equation*}
    so $(\mu^\ep_t,\psi^\ep_t,h^\ep_t)_{t\in[0,1]}\in \mathrm{CE}_p(\mu_0^\ep,\mu_1^\ep;[0,1])$.
Using \eqref{eq_orthogonal_decomposition_gmu} and the definition of $\Psi_t^\ep$, we have
    \begin{equation*}
    \begin{aligned}
    \langle \nabla\psi_t^\ep,\nabla\psi_t^\ep\rangle_{\mu_t^\ep}
    &=\langle \mathscr P_{\mu_t^\ep}\Psi_t^\ep,\mathscr P_{\mu_t^\ep}\Psi_t^\ep\rangle_{\mu_t^\ep}
    \le \langle \Psi_t^\ep,\Psi_t^\ep\rangle_{\mu_t^\ep}\\
    &=\frac{1}{2}\sum_{x,y\in\al X}\Psi_t^\ep(x,y)^2\,\hat \mu_t^\ep(x,y)\,K(x,y)\,\varpi(x)
    \\
    &=\frac{1}{2}\sum_{x,y\in\al X} 
    \frac{[(1-\ep)\nabla \psi_t(x,y)\hat\mu_t(x,y) 
   ]^2}
   {\hat\mu_t^\ep(x,y)}
    K(x,y)\varpi(x)
    \\
    &=\frac{1}{2}\sum_{x,y\in\al X} 
    \frac{[(1-\ep)\nabla \psi_t(x,y)\hat\mu_t(x,y) 
    +\ep\nabla \mathbf{0}(x,y) \hat\mu_t(x,y)]^2}
    {\theta((1-\ep)\mu_t(x)+\ep, (1-\ep)\mu_t(y)+\ep)}
    K(x,y)\varpi(x).
    \end{aligned}
    \end{equation*}
Here $\mathbf{0}=(0,0,\dots,0)\in\mathbb{R}^{\al X}$.
It is easy to verify the convexity of the function
    \begin{equation*}
    \mathbb{R}\times[0,+\infty) \times [0,+\infty)
    \ni (x,s,t) = \frac{x^2}{\theta(s,t)},
    \end{equation*}
    with the convention that $0^2/0=0$ and $x^2/0=+\infty$ if $x\neq 0$.
This implies
    \begin{align*}
    \langle \nabla\psi_t^\ep,
    \nabla\psi_t^\ep\rangle_{\mu_t^\ep}
    \le & 
    \frac{1-\ep}{2}\sum_{x,y\in\al X} 
    \frac{[\nabla \psi_t(x,y)\hat\mu_t(x,y)]^2}
    {\hat\mu_t(x,y)}
    K(x,y)\varpi(x)
    +
    \frac{\ep}{2}\sum_{x,y\in\al X} 
    \frac{0^2}
    {\theta(1,1)}
    K(x,y)\varpi(x)
    \\
    = &
    (1-\ep) \langle \nabla\psi_t,
    \nabla\psi_t\rangle_{\mu_t}.
    \end{align*}
These convexity arguments give
    \begin{equation}
    \label{ineq_positive_admissible_set_1}
    \begin{aligned}
    ({E}^{a,b}_{\rm LinSq}((\mu^{\varepsilon}_t,V^{\varepsilon}_t,h^{\varepsilon}_t)_{t\in[0,1]}))^{\frac{1}{2}}=&
    \int_{0}^{1}\!\sqrt{a^{2}|h_t^{\varepsilon}|^{2}+b^{2}\langle \nabla\psi_t^\ep,
    \nabla\psi_t^\ep\rangle_{\mu_t^\ep}}\,\di t
    \\
    \le & \int_{0}^{1}\!\sqrt{(1-\varepsilon)\big(a^{2}|h_t|^{2}+b^{2} \langle \nabla\psi_t,
    \nabla\psi_t\rangle_{\mu_t}\big)}\,\di t \\
    \le &\sqrt{1-\varepsilon}\,\big(\al W^{a,b}_{p}(\mu_0,\mu_1)+\varepsilon\big).
    \end{aligned}
    \end{equation}

Next, connect $\mu_i$ to $\mu_i^{\varepsilon}$ by the linear interpolation
    \begin{equation*}
    \mu_t^{i,\varepsilon}:=(1-t)\mu_i+t\,\mu_i^{\varepsilon},\qquad
    h_t^{i,\varepsilon}:=[\mu_i^{\varepsilon}-\mu_i,\varpi]\equiv \varepsilon\,[\mathbf{1}-\mu_i,\varpi],\quad \text{for} \ t\in[0,1].
    \end{equation*}
Since $\mu_t^{i,\varepsilon}\in \al M_{+}$ for $t\in(0,1]$, 
    Lemma~\ref{lem_ce_solver} enables us to define for a.e.\ $t\in[0,1]$,
    \begin{equation*}
    \psi_t^{i,\varepsilon}
    :=
    \big(B_{\mu_t^{i,\varepsilon}}\!\restriction_{\mathsf{H}}\big)^{-1}
    ( (\mu_i^{\varepsilon}-\mu_i)-h_t^{i,\ep}p)
    =
    \ep \big(B_{\mu_t^{i,\varepsilon}}\!\restriction_{\mathsf{H}}\big)^{-1}
    (\mathbf{1}-\mu_i-[\mathbf{1}-\mu_i,\varpi]p).
    \end{equation*}
So that 
    $(\mu^{i,\varepsilon}_t,\psi^{i,\varepsilon}_t,h^{i,\varepsilon}_t)_{t\in[0,1]}
    \in \mathrm{CE}_{p}(\mu_i,\mu_i^{\varepsilon};[0,1])$ 
    and
    \begin{equation}
    \label{ineq_positive_admissible_set_2}
    \begin{aligned}
    &\ ({E}^{a,b}_{\rm LinSq}((\mu^{i,\ep}_t,\psi^{i,\ep}_t,h^{i,\ep}_t)_{t\in[0,1]}))^{\frac{1}{2}}
    \\
    =&
    \int_0^1
    \sqrt{
    a^2 \left|h^{i,\ep}_t\right|^2
    +
    b^2
    \left[A_{\mu^{i,\ep}_t}\psi^{i,\ep}_t
    ,\psi^{i,\ep}_t
    \right]
    }
    \di t
    \\
    \le&
    \int_0^1
    \sqrt{
    a^2  |[\mu_i^\ep-\mu_i,\varpi]|^2
    +
    b^2
    \left\|A_{\mu^{i,\ep}_t}\right\|
    \left\|\psi^{i,\ep}_t \right\|_2^2
    }
    \di t
    \\
    \le&
    \int_0^1
    \sqrt{
    a^2 \ep^2  |[\mathbf{1}-\mu_i,\varpi]|^2
    +
    b^2 \ep^2
    \left\|A_{\mu^{i,\ep}_t}\right\|
    \left\|(B_{\mu^{i,\ep}_t }\restriction_{\mathsf H})^{-1}\right\|^2
    \left\|\mathbf{1}-\mu_i-[\mathbf{1}-\mu_i,\varpi]p\right\|_2^2
    }
    \di t
    \\
    \le &\
    C' \ep,
    \end{aligned}
    \end{equation}
    where $C'$ is a positive constant relying only on $\mu_0$ and $\mu_1$:
    \begin{equation*}
    \begin{aligned}
    C' := \max
    \bigg\{ &
    \sqrt{
    a^2  |[\mathbf{1}-\mu_i,\varpi]|^2
    +
    b^2
    \left\|A_{\nu}\right\|
    \left\|(B_{\nu}
    \restriction_{\mathsf H})^{-1}\right\|^2
    \left\| \mathbf{1}-\mu_i-[\mathbf{1}-\mu_i,\varpi]p \right\|_2^2
    }
    \\
    & \qquad \qquad \qquad \qquad \qquad \qquad \quad
    \bigg| \
    \|\nu\|_{\varpi,1}\le \|\mu_0\|_{\varpi,1}+\|\mu_1\|_{\varpi,1}+1,i=0,1
    \bigg\}.
    \end{aligned}
    \end{equation*}

We now concatenate the three pieces by a time reparametrization:
    \begin{equation*}
    (\bar\mu_t^{\varepsilon},\bar \psi_t^{\varepsilon},\bar h_t^{\varepsilon})
    :=
    \begin{cases}
    \big(\mu^{0,\varepsilon}_{3t},\,3\psi^{0,\varepsilon}_{3t},\,3h^{0,\varepsilon}_{3t}\big), & t\in[0,1/3],\\
    \big(\mu^{\varepsilon}_{3t-1},\,3\psi^{\varepsilon}_{3t-1},\,3h^{\varepsilon}_{3t-1}\big), & t\in(1/3,2/3),\\
    \big(\mu^{1,\varepsilon}_{3-3t},\,-3\psi^{1,\varepsilon}_{3-3t},\,-3h^{1,\varepsilon}_{3-3t}\big), & t\in[2/3,1].
    \end{cases}
    \end{equation*}
Then $(\bar\mu^{\varepsilon}_t,\bar \psi^{\varepsilon}_t,\bar h^{\varepsilon}_t)_{t\in[0,1]}\in \mathrm{CE}_{p}(\mu_0,\mu_1;[0,1])$ and $\bar\mu^{\varepsilon}_t$ is in $\al M_+$ at least for $t\in(0,1)$.
By Lemma \ref{lem_E_Linsq_tInvar}, 
    reparameterization invariance holds for ${E}^{a,b}_{\rm LinSq}$.
Gathering \eqref{ineq_positive_admissible_set_1} and \eqref{ineq_positive_admissible_set_2},
    \begin{align*}
    ({E}^{a,b}_{\rm LinSq}((\bar\mu^{\varepsilon}_t,\bar \psi^{\varepsilon}_t,\bar h^{\varepsilon}_t)_{t\in[0,1]}))^{1/2}
    =& ({E}^{a,b}_{\rm LinSq}((\mu^{0,\varepsilon}_t,\psi^{0,\varepsilon}_t,h^{0,\varepsilon}_t)_{t\in[0,1]}))^{1/2}
    + ( {E}^{a,b}_{\rm LinSq}((\mu^{\varepsilon}_t,\psi^{\varepsilon}_t,h^{\varepsilon}_t)_{t\in[0,1]}))^{1/2}
    \\
    &+ ( {E}^{a,b}_{\rm LinSq}((\mu^{1,\varepsilon}_t,\psi^{1,\varepsilon}_t,h^{1,\varepsilon}_t)_{t\in[0,1]}))^{1/2} \\
    \le& 2C'\,\varepsilon + \sqrt{1-\varepsilon}\,\big(\al W^{a,b}_{p}(\mu_0,\mu_1)+\varepsilon\big).
    \end{align*}
Letting $\varepsilon\downarrow 0$ gives
    \begin{equation*}
    \begin{aligned}
    &
    \inf\big\{( {E}^{a,b}_{\rm LinSq}((\mu_t,\psi_t,h_t)_{t\in[0,1]}))^{1/2} \mid (\mu_t,\psi_t,h_t)_{t\in[0,1]}\in \mathrm{CE}_{p}(\mu_0,\mu_1;[0,1]),\ \mu_t\in \al M_+\text{ on }(0,1)\big\}
    \\
     \le&
    \liminf_{\ep\to 0} ({E}^{a,b}_{\rm LinSq}((\bar\mu^{\varepsilon}_t,\bar \psi^{\varepsilon}_t,\bar h^{\varepsilon}_t)_{t\in[0,1]}))^{1/2}
    \le \al W^{a,b}_{p}(\mu_0,\mu_1).
    \end{aligned}
    \end{equation*}
The reverse inequality is immediate since restricting the feasible set can only increase the infimum.
Thus the formulation \eqref{eq_positive_admissible_set_psi} claimed in Lemma~\ref{lem_positive_admissible_set} holds.

\subsection{Proof of Lemma \ref{lem_differentiation_maximum}}
\label{sec_app_h}

For $t\in J, x_t\in I(t)$ and $h>0$ small,
\begin{equation*}
  \zeta(t+h)=\min_{x}\eta_{t+h}(x)\ \le\ \eta_{t+h}(x_t)
  = \eta_t(x_t)+h\,\dot\eta_t(x_t)+o(h)
  = \zeta(t)+h\,\dot\eta_t(x_t)+o(h).
\end{equation*}
Taking the minimum over $x_t\in I(t)$ and dividing by $h$ gives
\begin{equation*}
  \limsup_{h\downarrow 0}\frac{\zeta(t+h)-\zeta(t)}{h}
  \ \le\ \min_{x_t\in I(t)}\dot\eta_t(x_t).
\end{equation*}

For the reverse bound, pick $x_{t+h}\in I(t+h)$ so that $\zeta(t+h)=\eta_{t+h}(x_{t+h})$.
Then
\begin{equation*}
  \zeta(t+h)-\zeta(t)
  = \eta_{t+h}(x_{t+h})-\zeta(t)
  = \eta_t(x_{t+h})-\zeta(t) + h\,\dot\eta_t(x_{t+h}) + o(h)
  \ge h\,\dot\eta_t(x_{t+h}) + o(h),
\end{equation*}
since $\eta_t(x_{t+h})\ge \zeta(t)$.
Dividing by $h$ and taking $\liminf$,
\begin{equation*}
  \liminf_{h\downarrow 0}\frac{\zeta(t+h)-\zeta(t)}{h}
  \ \ge\ 
  \liminf_{h\downarrow 0}\dot\eta_t(x_{t+h}).
\end{equation*}
Take a sequence $h_k\downarrow 0$ such that
    \begin{equation*}
    \liminf_{h\downarrow 0}\dot\eta_t(x_{t+h})
    =
    \lim_{k\to\infty} \dot\eta_t(x_{t+h_k}).
    \end{equation*}
Because $\al X$ is finite, there exists a subsequence still denoted as $h_k$ such that $x_{t+h_k}=\bar x\in \al X$ for all large $k$.
By continuity of $\eta$ and $\zeta$,
    \begin{equation*}
    \eta_t(\bar x)
    =\lim_{k\to\infty} \eta_{t+h_k}(\bar x)
    =\lim_{k\to\infty}\zeta(t+h_k)
    =\zeta(t),
    \end{equation*}
    and hence $\eta_t(\bar x)=\zeta(t)$, 
    so $\bar x\in I(t)$.
Consequently,
\begin{equation*}
  \liminf_{h\downarrow 0}\frac{\zeta(t+h)-\zeta(t)}{h}
  \ \ge\ \lim_{k\to\infty} \dot\eta_t(x_{t+h_k})
  \ = \
  \dot\eta_t(\bar x)
  \ \ge\ \min_{x_t\in I(t)}\dot\eta_t(x_t).
\end{equation*}
Combining with the $\limsup$ estimate yields
    $\zeta'_+(t)=\min_{x_t\in I(t)}\dot\eta_t(x_t)$ for all  $t\in J$.

\subsection{Proof of Lemma \ref{lem_eigen_K}}
\label{sec_app_eigen_K}
On $\mathrm L^2(\al X,\varpi)$ we use 
    $\langle \psi,\xi\rangle_{\varpi}:=\sum_{x\in \al X}\psi(x)\,\xi(x)\,\varpi(x)$ 
    and $\| \psi\|_{\varpi,2}^2:=\langle \psi,\psi\rangle_{\varpi} $.
For $\psi,\xi:X\to\mathbb{R}$,
    \begin{equation*}
    \langle \psi, K\xi\rangle_{\varpi}
    = \sum_{x\in \al X}\psi(x)\Big(\sum_{y\in \al X}K(x,y)\,\xi(y)\Big)\varpi(x)
    = \sum_{x,y\in \al X}\psi(x)\,\xi(y)\,K(x,y)\,\varpi(x).
    \end{equation*}
By reversibility, $K(x,y)\varpi(x)=K(y,x)\varpi(y)$, hence
    \begin{equation*}
    \langle \psi, K\xi\rangle_{\varpi}
    = \sum_{x,y\in \al X}\psi(x)\,\xi(y)\,K(y,x)\,\varpi(y)
    = \sum_{y\in \al X}\xi(y)\Big(\sum_{x\in \al X}K(y,x)\,\psi(x)\Big)\varpi(y)
    = \langle K\psi, \xi\rangle_{\varpi}.
    \end{equation*}
Thus $K$ is self–adjoint on $\mathrm L^2(\varpi)$ and, in particular, its eigenvalues are real.

Next, since $K$ is row–stochastic, $K\mathbf{1}=\mathbf{1}$, so $1$ is an eigenvalue with eigenvector $\mathbf{1}$.
We now show that no eigenvalue exceeds $1$.
By Jensen’s inequality, for each $x\in \al X$,
\begin{equation*}
  (K\psi(x))^2
  = \Big(\sum_{y}K(x,y)\,\psi(y)\Big)^2
  \le \sum_{y}K(x,y)\,\psi(y)^2
  = K(\psi^2)(x).
\end{equation*}
Summing against $\varpi$ and using stationarity $\varpi^{\rm T} K= \varpi^{\rm T}$,
\begin{equation*}
\begin{aligned}
  \|K\psi\|_{2,\varpi}^2
  = \sum_{x}\varpi(x)\,(K\psi(x))^2
  \le& \sum_{x}\varpi(x)\,K(\psi^2)(x)
  \\
  =& \sum_{y}\psi(y)^2\,\Big(\sum_{x}\varpi(x)K(x,y)\Big)
  = \sum_{y}\psi(y)^2\,\varpi(y)
  = \|\psi\|_{2,\varpi}^2.
\end{aligned}
\end{equation*}
Hence the operator norm satisfies $\|K\|_{\mathrm{op}}\le 1$ on $\mathrm L^2(\varpi)$.
Since $K\mathbf{1}=\mathbf{1}$, we also have $\|K\|_{\mathrm{op}}\ge \|K\mathbf{1}\|_{2,\varpi}/\|\mathbf{1}\|_{2,\varpi}=1$.
Therefore $\|K\|_{\mathrm{op}}=1$.
For a self–adjoint operator, the operator norm equals the maximal eigenvalue in absolute value; combined with the existence of the eigenpair $(1,\mathbf{1})$, this shows that the largest eigenvalue of $K$ is $1$.

\bibliographystyle{plain}
\bibliography{ref}

@article{maas2011gradient,
  title   = {Gradient flows of the entropy for finite {Markov} chains},
  author  = {Maas, Jan},
  journal = {J. Funct. Anal.},
  volume  = {261},
  number  = {8},
  pages   = {2250--2292},
  year    = {2011}
}

@article{erbar2012ricci,
  title={{Ricci} curvature of finite {Markov} chains via convexity of the entropy},
  author={Erbar, Matthias and Maas, Jan},
  journal={Arch. Ration. Mech. Anal.},
  volume={206},
  pages={997--1038},
  year={2012},
  publisher={Springer}
}

@article{piccoli2016properties,
  title={On properties of the generalized {Wasserstein} distance},
  author={Piccoli, Benedetto and Rossi, Francesco},
  journal={Arch. Ration. Mech. Anal.},
  volume={222},
  pages={1339--1365},
  year={2016},
  publisher={Springer}
}

@article{piccoli2014generalized,
  title={Generalized {Wasserstein} distance and its application to transport equations with source},
  author={Piccoli, Benedetto and Rossi, Francesco},
  journal={Arch. Ration. Mech. Anal.},
  volume={211},
  pages={335--358},
  year={2014},
  publisher={Springer}
}

@article{dolbeault2009new,
  title={A new class of transport distances between measures},
  author={Dolbeault, Jean and Nazaret, Bruno and Savar{\'e}, Giuseppe},
  journal={Calc. Var. Partial Differ. Equ.},
  volume={34},
  number={2},
  pages={193--231},
  year={2009},
  publisher={Springer}
}

@article{figalli2010new,
  title={A new transportation distance between non-negative measures, with applications to gradients flows with {Dirichlet} boundary conditions},
  author={Figalli, Alessio and Gigli, Nicola},
  journal={J. Math. Pures Appl.},
  volume={94},
  number={2},
  pages={107--130},
  year={2010},
  publisher={Elsevier}
}

@article{chow2012fokker,
  title={{Fokker--Planck} equations for a free energy functional or {Markov} process on a graph},
  author={Chow, Shui-Nee and Huang, Wen and Li, Yao and Zhou, Haomin},
  journal={Arch. Ration. Mech. Anal.},
  volume={203},
  pages={969--1008},
  year={2012},
  publisher={Springer}
}

@article{chow2018entropy,
  title={Entropy dissipation of {Fokker--Planck} equations on graphs},
  author={Chow, Shui-Nee and Li, Wuchen and Zhou, Haomin},
  journal={Discrete Contin. Dyn. Syst.},
  volume={38},
  number={10},
  pages={4929--4950},
  year={2018},
  publisher={Discrete Contin. Dyn. Syst.}
}

@article{liero2018optimal,
  title={Optimal entropy-transport problems and a new {Hellinger--Kantorovich} distance between positive measures},
  author={Liero, Matthias and Mielke, Alexander and Savar{\'e}, Giuseppe},
  journal={Invent. Math.},
  volume={211},
  number={3},
  pages={969--1117},
  year={2018},
  publisher={Springer}
}

@article{jordan1998variational,
  title={The variational formulation of the {Fokker--Planck} equation},
  author={Jordan, Richard and Kinderlehrer, David and Otto, Felix},
  journal={SIAM J. Math. Anal. },
  volume={29},
  number={1},
  pages={1--17},
  year={1998},
  publisher={SIAM}
}

@article{benamou2000computational,
  title={A computational fluid mechanics solution to the {Monge--Kantorovich} mass transfer problem},
  author={Benamou, Jean-David and Brenier, Yann},
  journal={Numer. Math.},
  volume={84},
  number={3},
  pages={375--393},
  year={2000},
  publisher={Springer-Verlag Berlin/Heidelberg}
}

@article{erbar2022gradient,
    title = {Gradient flow formulation of diffusion equations in the {Wasserstein} space over a metric graph},
    journal = {Netw. Heterog. Media},
    volume = {17},
    number = {5},
    pages = {687-717},
    year = {2022},
    issn = {1556-1801},
    doi = {10.3934/nhm.2022023},
    url = {https://www.aimsciences.org/article/id/62aabe4d2d80b75d2b6d88ce},
    author = {Matthias Erbar and Dominik Forkert and Jan Maas and Delio Mugnolo}
}

@article{chizat2018interpolating,
  title={An interpolating distance between optimal transport and {Fisher--Rao} metrics},
  author={Chizat, Lenaic and Peyr{\'e}, Gabriel and Schmitzer, Bernhard and Vialard, Fran{\c{c}}ois-Xavier},
  journal={Found. Comput. Math.},
  volume={18},
  number={1},
  pages={1--44},
  year={2018},
  publisher={Springer}
}

@article{kondratyev2016new,
  title={A new optimal transport distance on the space of finite {Radon} measures},
  author={Kondratyev, Stanislav and Monsaingeon, L{\'e}onard and Vorotnikov, Dmitry},
  journal={Adv. Differ. Equ.},
  volume={21},
  number={11-12},
  pages={1117--1164},
  year={2016},
  publisher={Khayyam Publishing, Inc}
}

@article{gangbo2019unnormalized,
  title={Unnormalized optimal transport},
  author={Gangbo, Wilfrid and Li, Wuchen and Osher, Stanley and Puthawala, Michael},
  journal={J. Comput. Phys.},
  volume={399},
  pages={108940},
  year={2019},
  publisher={Elsevier}
}

@article{erbar2019geometry,
  title={On the geometry of geodesics in discrete optimal transport},
  author={Erbar, Matthias and Maas, Jan and Wirth, Melchior},
  journal={Calc. Var. Partial Differ. Equ.},
  volume={58},
  number={1},
  pages={19},
  year={2019},
  publisher={Springer}
}

@article{mielke2013geodesic,
  title={Geodesic convexity of the relative entropy in reversible {Markov} chains},
  author={Mielke, Alexander},
  journal={Calc. Var. Partial Differ. Equ.},
  volume={48},
  number={1},
  pages={1--31},
  year={2013},
  publisher={Springer}
}

@article{mielke2011gradient,
  title={A gradient structure for reaction-diffusion systems and for energy-drift-diffusion systems},
  author={Mielke, Alexander},
  journal={Nonlinearity},
  volume={24},
  number={4},
  pages={1329},
  year={2011},
  publisher={IOP Publishing}
}

@article{che2016convergence,
  title={Convergence to global equilibrium for {Fokker--Planck} equations on a graph and {Talagrand-type} inequalities},
  author={Che, Rui and Huang, Wen and Li, Yao and Tetali, Prasad},
  journal={Journal of Differential Equations},
  volume={261},
  number={4},
  pages={2552--2583},
  year={2016},
  publisher={Elsevier}
}

@article{CarrilloE2025,
  title={Evolution equations on co-evolving graphs: long-time behaviour and the graph continuity equation},
  author={José Antonio Carrillo and Esposito, Antonio and Mikolás László },
  journal={arXiv preprint arXiv:2504.10446},
  year={2025}
}

@article{esposito2021nonlocal,
  title={Nonlocal-interaction equation on graphs: gradient flow structure and continuum limit},
  author={Esposito, Antonio and Patacchini, Francesco S and Schlichting, Andr{\'e} and Slep{\v{c}}ev, Dejan},
  journal={Arch. Ration. Mech. Anal.},
  volume={240},
  number={2},
  pages={699--760},
  year={2021},
  publisher={Springer}
}

@article{li2022lojasiewicz,
  title={The {{\L}ojasiewicz} inequality for free energy functionals on a graph.},
  author={Li, Kongzhi and Xue, Xiaoping},
  journal={Commun. Pure Appl. Anal.},
  volume={21},
  number={8},
  pages = {2661--2677},
  year={2022}
}

@article{li2025gradient,
  title={Gradient flows of generalized relative entropy and functional inequalities on graphs},
  author={Li, Kongzhi and Xue, Xiaoping},
  journal={J. Math. Anal. Appl.},
  volume={543},
  number={1},
  pages={128862},
  year={2025},
  publisher={Elsevier}
}

@book{ambrosio2005gradient,
  title={Gradient flows: in metric spaces and in the space of probability measures},
  author={Ambrosio, Luigi and Gigli, Nicola and Savar{\'e}, Giuseppe},
  year={2005},
  publisher={Springer}
}

@article{otto2001geometry,
  title={The geometry of dissipative evolution equations: the porous medium equation},
  author={Otto, Felix},
  journal={Commun. Partial Differ. Equ.},
  volume={26},
  number = {1-2}, 
  pages = {101--174},
  year={2001}
}

@inproceedings{le2023scalable,
  title     = {Scalable unbalanced {Sobolev} transport for measures on a graph},
  author    = {Le, Tam and Nguyen, Truyen and Fukumizu, Kenji},
  booktitle = {Proceedings of The 26th International Conference on Artificial Intelligence and Statistics},
  pages     = {8521--8560},
  year      = {2023},
  editor    = {Ruiz, Francisco and Dy, Jennifer and van de Meent, Jan-Willem},
  volume    = {206},
  series    = {Proceedings of Machine Learning Research},
  publisher = {PMLR}
}

@article{keriven2023entropic,
  title={Entropic optimal transport on random graphs},
  author={Keriven, Nicolas},
  journal={SIAM J. Math. Data Sci.},
  volume={5},
  number={4},
  pages={1028--1050},
  year={2023},
  publisher={SIAM}
}

@book{villani2008optimal,
  title={Optimal transport: old and new},
  author={Villani, C{\'e}dric and others},
  volume={338},
  year={2008},
  publisher={Springer}
}

@book{villani2003topics,
  title     = {Topics in Optimal Transportation},
  author    = {Villani, C{\'e}dric},
  volume    = {58},
  year      = {2003},
  publisher = {American Mathematical Society}
}

@article{lovasz1993random,
  title   = {Random walks on graphs: A survey},
  author  = {Lov{\'a}sz, L{\'a}szl{\'o}},
  journal = {Combinatorics, Paul Erd{\H{o}}s is Eighty},
  volume  = {2},
  pages   = {1--46},
  year    = {1993},
  publisher = {J{\'a}nos Bolyai Mathematical Society}
}

@article{CarrilloWang2025,
  title={{Fokker--Planck} equations on discrete infinite graphs},
  author={José Antonio Carrillo and Xinyu Wang},
  journal={arXiv preprint arXiv:2511.12076},
  year={2025}
}

@book{chung1997spectral,
  title={Spectral graph theory},
  author={Chung, Fan R. K.},
  volume={92},
  year={1997},
  publisher={American Mathematical Society}
}

@article{farahani2013review,
  title={A review of urban transportation network design problems},
  author={Farahani, Reza Zanjirani and Miandoabchi, Elnaz and Szeto, Wai Yuen and Rashidi, Hannaneh},
  journal={Eur. J. Oper. Res.},
  volume={229},
  number={2},
  pages={281--302},
  year={2013},
  publisher={Elsevier}
}

@article{pastor2015epidemic,
  title={Epidemic processes in complex networks},
  author={Pastor-Satorras, Romualdo and Castellano, Claudio and Van Mieghem, Piet and Vespignani, Alessandro},
  journal={Rev. Mod. Phys.},
  volume={87},
  number={3},
  pages={925--979},
  year={2015},
  publisher={APS}
}

@article{zuniga2020reaction,
  title={Reaction--diffusion equations on complex networks and {Turing} patterns, via p-adic analysis},
  author={Z{\'u}{\~n}iga-Galindo, W. A.},
  journal={J. Math. Anal. Appl.},
  volume={491},
  number={1},
  pages={124239},
  year={2020},
  publisher={Elsevier}
}

@article{weber2006multicomponent,
  title={Multicomponent reaction-diffusion processes on complex networks},
  author={Weber, Sebastian and Porto, Markus},
  journal={Phys. Rev. E Stat. Nonlin. Soft Matter Phys.},
  volume={74},
  number={4},
  pages={046108},
  year={2006},
  publisher={APS}
}

@article{mao2026wasserstein2,
  title={{Wasserstein} geometry of nonnegative measures on finite {Markov} chains {II}: Geodesic and duality formulae},
  author={Mao, Qifan and Wang, Xinyu and Xue, Xiaoping},
  journal={Submitted}
}

\end{document}